\documentclass{amsart}

\usepackage[T1]{fontenc}
\usepackage[a4paper,top=46mm,bottom=46mm,left=37mm,right=37mm]{geometry}

\usepackage{latexsym,amsfonts} 
\usepackage{amsmath,amssymb,graphics,setspace}
\usepackage{amsthm}
\usepackage{amscd} 
\usepackage{MnSymbol}
\usepackage{mathrsfs} 
\usepackage{fontenc}%

\usepackage{marvosym}
\usepackage{paralist}
\usepackage{color}

\usepackage{pstricks,pst-text,pst-grad,pst-node,pst-3dplot,pstricks-add,pst-poly} 
\usepackage{pst-solides3d}

\usepackage{graphicx}
\usepackage{fontenc}%

\usepackage{hyperref}  
\hypersetup{linktocpage=true,colorlinks=true,linkcolor=blue,citecolor=blue,pdfstartview={XYZ 1000 1000 1}}

\usepackage[verbose]{wrapfig}

\usepackage{subfigure}
\renewcommand{\thesubfigure}{\thefigure.\arabic{subfigure}}
\makeatletter
\renewcommand{\p@subfigure}{}
\renewcommand{\@thesubfigure}{\thesubfigure:\hskip\subfiglabelskip}
\makeatother

\newtheorem{theorem}{Theorem}
\newtheorem{lemma}{Lemma}
\newtheorem{corollary}{Corollary}
\newtheorem{proposition}{Proposition}

\newtheorem{remark}{Remark}
\theoremstyle{definition}
\newtheorem{definition}{Definition}

\newtheorem{example}{Example}
\newtheorem{conjecture}{Conjecture}

\newcommand{\norm}[1]{\left\|#1\right\|}
\newcommand{\cl}{\mbox{cl}} 
\newcommand{\re}{\mbox{re}} 
\newcommand{\Ree}{\mbox{Re}} 
\newcommand{\area}{\mbox{area}} 
\newcommand{\shape}{\mbox{shape}} 
\newcommand{\curvature}{\mbox{curvature}} 
\newcommand{\perim}{\mbox{perim}} 
\newcommand{\cv}{\mbox{conv}} 
\newcommand{\sh}{\mbox{sheet}} 
\newcommand{\str}{\mbox{str}} 
\newcommand{\strShape}{\mbox{strShape}} 
\newcommand{\wsh}{\mbox{wsh}} 
\newcommand{\wshShape}{\mbox{wshShape}} 
\newcommand{\near}{\delta} 
\newcommand{\dnear}{\delta_{\Phi}} 
\newcommand{\dfar}{{\not\delta}_{\Phi}} 
\newcommand{\dcap}{\mathop{\cap}\limits_{\Phi}} 
\newcommand{\sn}{\mathop{\delta}\limits^{\doublewedge}} 
\newcommand{\snd}{\mathop{\delta_{_{\Phi}}}\limits^{\doublewedge}} 
\newcommand{\Int}{\mbox{int}}
\newcommand{\sdfar}{\stackrel{\not{\text{\normalsize$\delta$}}}{\text{\tiny$\doublevee$}}_{_{\Phi}}} 

\newcommand{\sfar}{\stackrel{\not{\text{\normalsize$\delta$}}}{\text{\tiny$\doublevee$}}} 

\newcommand{\bdy}{\mbox{bdy}}     
\newcommand{\bnd}{\mbox{bnd}_{\Phi}}     
\newcommand{\nhbd}{\mbox{nhbd}}     

\begin{document}

\def\defineTColor#1#2{%
 \newpsstyle{#1}{%
  fillstyle=vlines,hatchcolor=#2,
  hatchwidth=0.1\pslinewidth,
  hatchsep=1\pslinewidth}%
  }
\defineTColor{Tgray}{gray}
\defineTColor{Tgreen}{green}
\defineTColor{Tyellow}{yellow}
\defineTColor{Tred}{red} 
\defineTColor{Tblue}{blue} 
\defineTColor{Tmagenta}{magenta}
\defineTColor{Torange}{orange}
\defineTColor{Twhite}{white}
\defineTColor{Tbgreen}{brightgreen}

\title[Two Forms of Proximal Physical Geometry]{Two Forms of Proximal Physical Geometry.\\  Axioms, Sewing Regions Together, Classes of Regions, Duality, and Parallel Fibre Bundles}

\author[J.F. Peters]{J.F. Peters$^{\alpha}$}
\email{James.Peters3@umanitoba.ca}
\address{\llap{$^{\alpha}$\,}Computational Intelligence Laboratory,
University of Manitoba, WPG, MB, R3T 5V6, Canada and
Department of Mathematics, Faculty of Arts and Sciences, Ad\.{i}yaman University, 02040 Ad\.{i}yaman, Turkey}
\thanks{The research has been supported by the Natural Sciences \&
Engineering Research Council of Canada (NSERC) discovery grant 185986 
and Instituto Nazionale di Alta Matematica (INdAM) Francesco Severi, Gruppo Nazionale per le Strutture Algebriche, Geometriche e Loro Applicazioni grant 9 920160 000362, n.prot U 2016/000036.}


\date{}

\dedicatory{Dedicated to Som Naimpally and V.F. Lenzen}

\begin{abstract}
This paper introduces two proximal forms of Lenzen physical geometry, namely, an \emph{axiomatized strongly proximal physical geometry} that is built on simplicial complexes with the dualities and sewing operations derived from string geometry and an \emph{axiomatized descriptive proximal physical geometry} in which spatial regions are described based on their features and the descriptive proximities between regions.  This is a computational proximity approach to a Lenzen geometry of physical space.  In both forms of physical geometry, region is a primitive.  Intuitively, a region is a set of connected subregions.   The primitive in this geometry is \emph{region}, instead of \emph{point}.  Each description of a region with $n$ features is a mapping from the region to a feature vector in $\mathbb{R}^n$.  In the feature space, proximal physical geometry has the look and feel of either Euclidean, Riemannian, or non-Euclidean geometry, since we freely work with the relations between points in the feature space. The focus in these new forms of geometry is the relation between individual regions with their own distinctive features such as shape, area, perimeter and diameter and the relation between nonempty sets of regions. The axioms for physical geometry as well as the axioms for proximal physical geometry are given and illustrated.  Results for parallel classes of regions, descriptive fibre bundles and BreMiller-Sloyer sheaves are given.  In addition, a region-based Borsuk-Ulam Theorem as well as a Wired Friend Theorem are given in the context of both forms of physical geometry.
\end{abstract}


\maketitle

\section{Introduction}
This paper introduces two new forms of Lenzen physical geometry~\cite{Lenzen1939AMMphysicalGeometry} in which regions have distinctive features, leading to mappings of regions into $\mathbb{R}^n$ and the proximities between regions as well as between nonempty sets of regions are its focus.

\setlength{\intextsep}{0pt}

\begin{wrapfigure}[18]{R}{0.55\textwidth}
\begin{minipage}{4.2 cm}
\centering
\includegraphics[width=65mm]{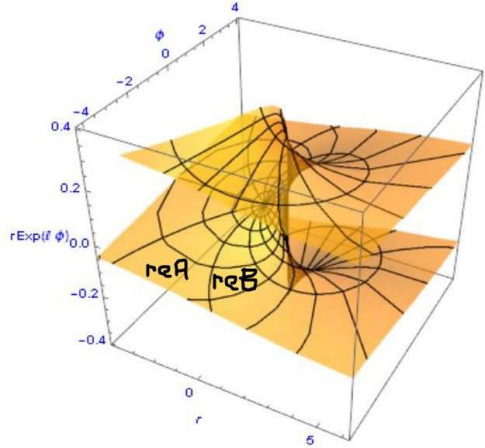}
\caption[]{\footnotesize\bf Overlapping Regions}
\label{fig:reAreB}
\end{minipage}
\end{wrapfigure}
$\mbox{}$\\
\vspace{5mm}

The primitive in this geometry is \emph{region}, instead of \emph{point}.  The result is a pure physical geometry (called \emph{strongly proximal physical geometry}) on a space of regions, which is pair $(Re,\sn)$, where $Re$ is a set of regions and $\sn$ is a strong proximity (overlap) relation on pairs of regions.   From this, it is then possible to give axioms for this form of physical geometry. 

In its descriptive form, a \emph{proximal physical geometric space} is a set of physical regions endowed with one or more descriptive proximity relations.  This descriptive physical geometric space is a pair$(Re,\mathscr{R})$, where $Re$ is a set of regions and $\mathscr{R}$ is a proximal relator, {\em i.e.}, a set of one or more proximity relations on $Re$ that minimally includes a descriptive proximity $\dnear$.  From this, it is then possible to give axioms for descriptive physical geometry. 

$\mbox{}$\\
\vspace{8mm}

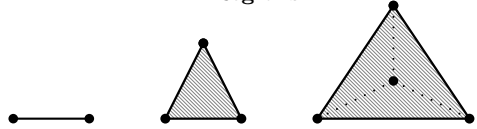
\begin{figure}[!ht]
\begin{center}
 \begin{pspicture}
 (-2.5,1.0)(4.5,2.5)
  \pspolygon[showpoints=true,linecolor=black,linestyle=solid,linewidth=0.03,style=Tgray] 
  (1.00,1.00)(2.00,1.00)(1.50,2.00)
	\psdot(1.50,2.00)
	\pspolygon[showpoints=true,linecolor=black,linestyle=solid,linewidth=0.03,style=Tgray]
  (3.00,1.00)(5.00,1.00)(4.00,2.50) 
	\psline[linestyle=dotted]{-}(3.0,1.0)(4.00,1.50)\psline[linestyle=dotted]{-}(4.00,1.50)(4.00,2.50)
	\psline[linestyle=dotted]{-}(4.00,1.50)(5.00,1.00)
	\psdot(4.00,2.50)\psdot(4.00,1.50)
	\psline[showpoints=true]{-}(-1.0,1.0)(0.0,1)
 \end{pspicture}
  \caption{1-, 2-, 3-Simplicial Complex}
	\label{fig:Simplices}
 \end{center}
\end{figure}
$\mbox{}$\\
\vspace{3mm}

In both types of physical geometry, the basic structures are vertex regions, edges (each edge has two vertices), triangles (each triangle has 3 vertices), facial regions (surfaces of polyhedrons), $k$-simplices, simplicial complexes, strings and worldsheets.  Every basic structure is collection of regions.  A $k$-\emph{simplex} is the smallest possible polytope in an $n$-dimensional space.  The canonical home for a simplex is a real Euclidean space~\cite[\S 3.2.1, p. 109]{Krantz2010Homology}, namely, $\mathbb{R}^n$.   Let $A$ be a triangle.  A \emph{convex hull} (denoted by $\cv A$) is the smallest convex set containing $A$.
A \emph{simplicial complex} is a finite collection of simplices $K$ such that if $\re A \in K$ and $\re B\in \re A$, then $\re A \in K$ (boundedness condition).  For example, in Fig.~\ref{fig:Simplices},
a $k$-\emph{simplicial complex} is an edge ($k$ = 1), convex hull of a triangle ($k$ = 2) and convex hull of a tetrahedron ($k$ = 3).

In addition, if $\re A, \re B\in K$, then $\re A\cap \re B = \emptyset$ or $\re A\cap \re B$ is a face of both $\re A$ and $\re B$~\cite{Edelsbrunner1999}.  A {\bf $k$-simplex} (or {\bf k-simplicial complex}) has $v_0,v_1,\dots,v_d$ vertices and $k-1$ faces in $\mathbb{R}^d$. A $k$-simplex is the convex hull of $k+1$ {\bf linearly independent} points (vectors).  Let $X\subset \mathbb{R}^k$.  Recall that a {\bf convex hull} of $X$ (denoted $\mbox{conv}\ X$) is defined by
\[
\mbox{conv}\ X = \bigcap\left\{C: C\ \mbox{is a convex set that contains}\ X\right\}.
\]
A set of vectors $v_0,v_1,\dots,v_k$ in a $k$-dimensional space are linearly independent over a real field $F$, if and only if, for all scalars $\alpha_i\in F$,
\[
\alpha_0v_0 + \alpha_1v_1,\dots,\alpha_kv_k = 0\ \mbox{implies}\ \alpha_0 = \alpha_1 = \dots = \alpha_k = 0.
\]
A 0-simplex (vertex), 1-simplex (line segment), 2-simplex (solid triangle) and 3-simplex (3-sided polyhedron) are represented in Fig.~\ref{fig:Simplices}. 

Two important structures in physical geometry are strings and world sheets.  Recall that a geometric structure that has the characteristics of a cosmological string $A$ (denoted $\str A$), which is the path followed by a particle $A$ moving through space.  A \emph{string} is a region of space with non-zero width (in a non-abstract, physical geometry space such as the space described by V.F. Lenzen~\cite{Lenzen1939AMMphysicalGeometry}) and has either bounded or unbounded length.  Another name for such a string is \emph{worldline}~\cite{OliveLandsberg1989stringTheory,Olive1988stringsAndSuperstrings,Olive1987algebrasAndStrings}.  Here, a string is a region on the surface of an $n$-sphere or a region in an $n$-dimensional normed linear space.  Every string is a spatial region, but not every spatial region is a string.

Let $M$ be a nonempty region of a space $X$.  Region $M$ is a \emph{worldsheet} (denoted by $\sh\ M$), provided every subregion of $\sh\ M$ contains at least one string.  In other words, a worldsheet is completely covered by strings.  Let $X$ be a collection of strings.  $X$ is a \emph{cover} for $\sh\ M$, provided $\sh\ M\subseteq X$.  Every member of a worldsheet is a string.

\begin{figure}[!ht]
\centering
\subfigure[bounded worldsheet]
 {\label{fig:1worldsheet}\includegraphics[width=85mm]{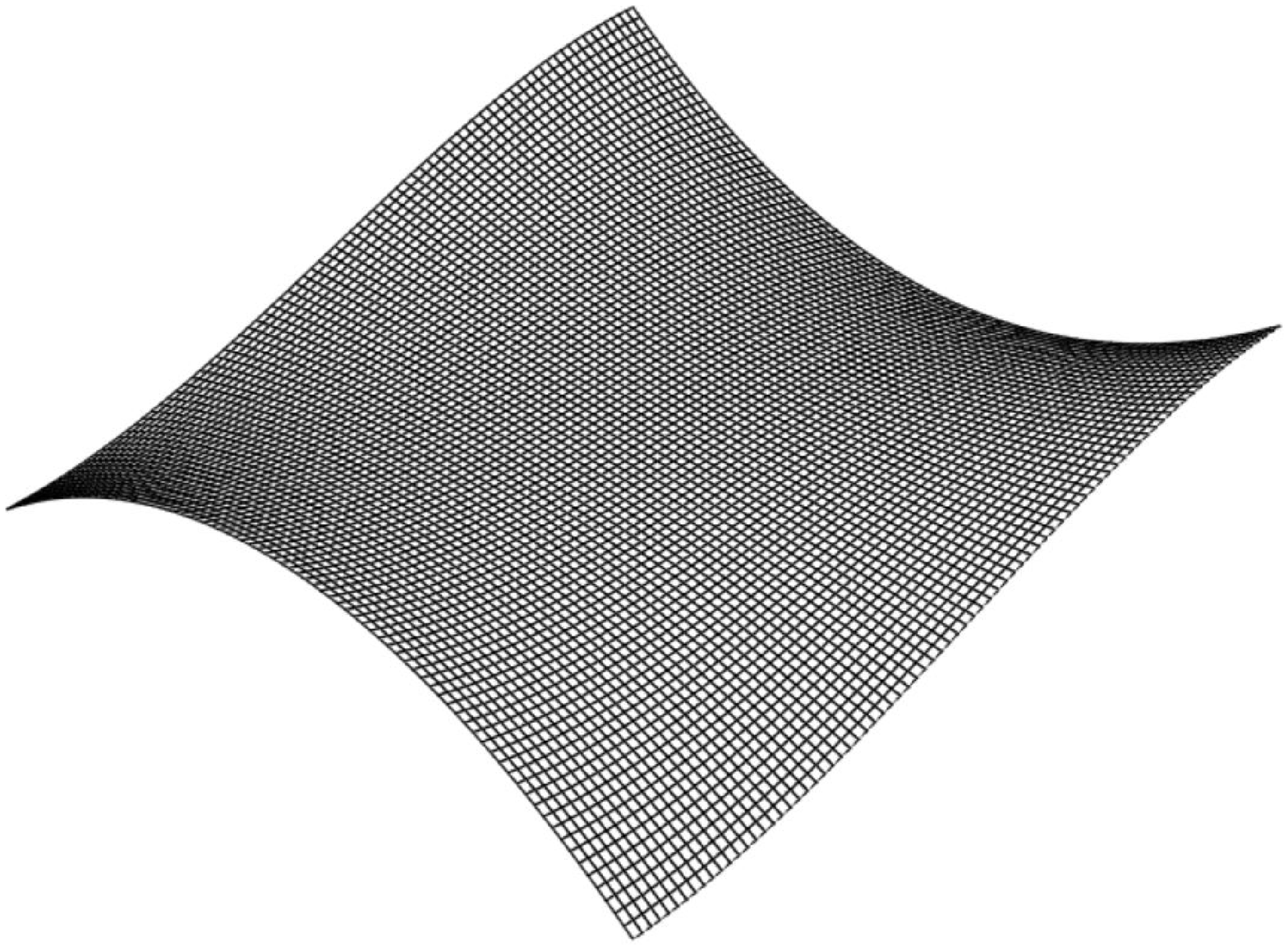}}\hfil
\subfigure[rolled worldsheet]
 {\label{fig:2worldsheet}\includegraphics[width=30mm]{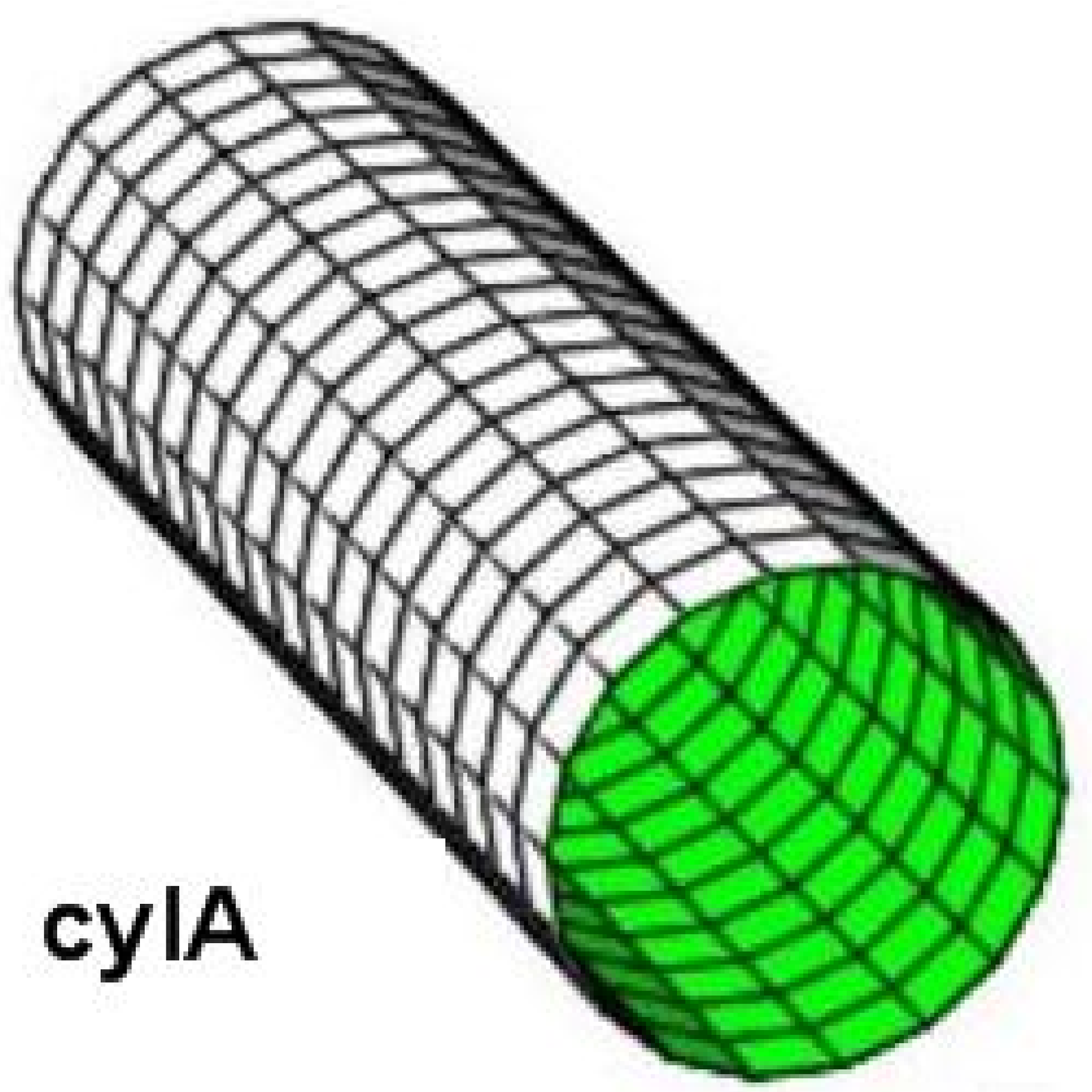}}
\caption[]{Bounded worldsheet $\mapsto$ cylinder surface}
\label{fig:corticalRegions}
\end{figure}
$\mbox{}$\\
\vspace{2mm}

\begin{example}\label{ex:rolledUpWorldsheet}
\includegraphics[width=30mm]{1worldsheet}{\large $\boldsymbol{\longmapsto}$}\includegraphics[width=10mm]{2worldsheet} A worldsheet $\sh M$ with finite width and finite length is represented in Fig.~\ref{fig:1worldsheet}.  This worldsheet is rolled up to form the lateral surface of a cylinder represented in Fig.~\ref{fig:2worldsheet}, namely, $2\pi rh$ with radius $r$ and height $h$ equal to the length of $\sh\ M$.  We call this a worldsheet cylinder.  In effect, a flattened, bounded worldsheet maps to a worldsheet cylinder. \qquad \textcolor{blue}{$\blacksquare$}
\end{example}

A \emph{line} is a region of space with zero width (in an abstract geometric space or geometry in a science of spaces~\cite[\S 8.0]{ScribaSchreiber2005geometry}, esp.~\cite{Brouwer1954CJMpointsAndSpaces}) or non-zero width (in a non-abstract, physical geometry space such as the geometry described by V.F. Lenzen~\cite{Lenzen1939AMMphysicalGeometry}) and has either bounded or unbounded length.  Each region in a physical geometry has a non-zero area.  Hence, a physical geometry space does not contain the usual points found in an abstract space.  A physical region is described by a feature vector in $\mathbb{R}^n$.  Each component of a feature vector is a probe function value for a region feature such as area, perimeter and diameter.  

This form of physical geometry is an example of a geometry without points.  

Many other forms of point-free geometry have been proposed.  For example, in the Gerla-Volpe point-free geometry~\cite{GerlaVolpe1985AMMpointfreeGeometry}, the primitives are solid, inclusion and the distance between solids.  Other examples of point-free geometries are given in which the primitives are regions, inclusion, minimum and maximum distance~\cite{Gerla2007JPLpointfreeGeometry}, or the primitives are region and quasi-metric~\cite{DiConcilio2013mcs} or the primitives are regions, closeness, smallness and inclusion~\cite{CoppolaGerla2015LLPpointfreeGeometry} (see, also,~\cite{Gerla1994pointlessGeometries,GerlaPaolillo2010WhiteheadGeometry,CoppolaPacelli2006pointlessGeometry,CoppolaGerla2015pointfreeGeometry,Gerla2007pointfreeGeometryVersimilitudes,Gerla2005pointfreeGeometryGradedInclusion,Gerla1982senzaPuntiGeometry}).

An example of a physical region is a string that is a \emph{worldline}~\cite{OliveLandsberg1989stringTheory,Olive1988stringsAndSuperstrings,Olive1987algebrasAndStrings}.  A string on the surface of an $n$-sphere is a line that represents the path traced by moving particle along the surface of the $S^n$.  Disjoint strings on the surface of $S^n$ that are antipodal and with matching description are descriptively near \emph{antipodal strings}.    A pair of strings $A,\righthalfcap A$ are \emph{antipodal}, provided, for some subregions $p\in A, q\in \righthalfcap A$, there exist disjoint parallel hyperplanes $P,Q\subset S^n$ such that $p\in P$ and $q\in Q$.   Such strings can be spatially far apart and also descriptively near.

\begin{example}
Geometric regions that share subregions are strongly near.  For example, many pairs of strongly near (overlapping) regions $\re A,\re B$ are on the curved surface of an inverse elliptical nome such as the one in Fig.~\ref{fig:reAreB}.   We write $\re A\ \sn\ \re B$, since this pair of regions has a common edge.
\qquad \textcolor{blue}{$\blacksquare$}
\end{example}

A physical geometric space is a proximal relator space~\cite{Peters2016relator}, provided the space is equipped with a family of proximity relations.
Such a space is descriptive, provided one or more of its proximities is descriptive~\cite{Peters2015AMSJmanifolds,PetersGuadagni2015strongConnectedness} (see, also,~\cite[\S 1.4, pp. 16-17, 26-29 and \S 4.3, pp. 128-132]{Peters2016ComputationalProximity}).  A descriptive physical geometry is a form of computational geometry in which, for example, the nearness of regions in a Vorono\"{i} tessellation of sets of regions is determined.

\section{Preliminaries}
This section briefly introduces spatial and descriptive forms of proximity that provide a basis for two corresponding forms of strong Lodato proximity introduced in~\cite{Peters2015AMSJmanifolds} and axiomatized in~\cite{Peters2016ComputationalProximity}.  

Let $X$ be a nonempty set. A \emph{Lodato proximity}~\cite{Lodato1962,Lodato1964,Lodato1966} $\delta$ is a relation on the family of sets $2^X$, which satisfies the following axioms for all subsets $A, B, C $ of $X$:\\

\begin{description}
\item[{\rm\bf (P0)}] $\emptyset \not\delta A, \forall A \subset X $.
\item[{\rm\bf (P1)}] $A\ \delta\ B \Leftrightarrow B \delta A$.
\item[{\rm\bf (P2)}] $A\ \cap\ B \neq \emptyset \Rightarrow A \near B$.
\item[{\rm\bf (P3)}] $A\ \delta\ (B \cup C) \Leftrightarrow A\ \delta\ B $ or $A\ \delta\ C$.
\item[{\rm\bf (P4)}] $A\ \delta\ B$ and $\{b\}\ \delta\ C$ for each $b \in B \ \Rightarrow A\ \delta\ C$. \qquad \textcolor{blue}{$\blacksquare$}
\end{description}
\vspace{3mm}
Further $\delta$ is \textit{separated }, if 
\vspace{3mm}
\begin{description}
\item[{\rm\bf (P5)}] $\{x\}\ \delta\ \{y\} \Rightarrow x = y$. \qquad \textcolor{blue}{$\blacksquare$}
\end{description}

\setlength{\intextsep}{0pt}

\noindent We can associate a topology with the space $(X, \delta)$ by considering as closed sets those sets that coincide with their own closure.  For a nonempty set $A\subset X$, the closure of $A$ (denoted by $\mbox{cl} A$) is defined by,
\[
\mbox{cl} A = \{ x \in X: x\ \delta\ A\}.
\]

The descriptive proximity $\delta_{\Phi}$ was introduced in~\cite{Peters2012ams}.   Let $A,B \subset X$ and let $\Phi(x)$ be a feature vector for $x\in X$, a nonempty set of non-abstract points such as picture points.  $A\ \delta_{\Phi}\ B$ reads $A$ is descriptively near $B$, provided $\Phi(x) = \Phi(y)$ for at least one pair of points, $x\in A, y\in B$.  From this, we obtain the description of a set and the descriptive intersection~\cite[\S 4.3, p. 84]{Naimpally2013} of $A$ and $B$ (denoted by $A\ \dcap\ B$) defined by
\begin{description}
\item[{\rm\bf ($\boldsymbol{\Phi}$)}] $\Phi(A) = \left\{\Phi(x)\in\mathbb{R}^n: x\in A\right\}$, set of feature vectors.
\item[{\rm\bf ($\boldsymbol{\dcap}$)}]  $A\ \dcap\ B = \left\{x\in A\cup B: \Phi(x)\in \Phi(A)\ \mbox{and}\ \Phi(x)\in \Phi(B)\right\}$.
\qquad \textcolor{blue}{$\blacksquare$}
\end{description}
Then swapping out $\near$ with $\dnear$ in each of the Lodato axioms defines a descriptive Lodato proximity. 

That is, a \textit{descriptive Lodato proximity $\dnear$} is a relation on the family of sets $2^X$, which satisfies the following axioms for all subsets $A, B, C $ of $X$.\\

\begin{description}
\item[{\rm\bf (dP0)}] $\emptyset\ \dfar\ A, \forall A \subset X $.
\item[{\rm\bf (dP1)}] $A\ \dnear\ B \Leftrightarrow B\ \dnear\ A$.
\item[{\rm\bf (dP2)}] $A\ \dcap\ B \neq \emptyset \Rightarrow\ A\ \dnear\ B$.
\item[{\rm\bf (dP3)}] $A\ \dnear\ (B \cup C) \Leftrightarrow A\ \dnear\ B $ or $A\ \dnear\ C$.
\item[{\rm\bf (dP4)}] $A\ \dnear\ B$ and $\{b\}\ \dnear\ C$ for each $b \in B \ \Rightarrow A\ \dnear\ C$. \qquad \textcolor{blue}{$\blacksquare$}
\end{description}
\vspace{3mm}
Further $\dnear$ is \textit{descriptively separated }, if 
\vspace{3mm}
\begin{description}
\item[{\rm\bf (dP5)}] $\{x\}\ \dnear\ \{y\} \Rightarrow \Phi(x) = \Phi(y)$ ($x$ and $y$ have matching descriptions). \qquad \textcolor{blue}{$\blacksquare$}
\end{description}
\vspace{3mm}

\begin{wrapfigure}[8]{R}{0.45\textwidth}
\begin{minipage}{4.5 cm}
\centering
\includegraphics[width=40mm]{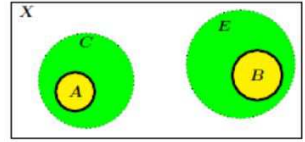}
\caption[]{\footnotesize\bf Strongly Far}
\label{fig:stronglyFar}
\end{minipage}
\end{wrapfigure}

\noindent The pair $\left( X,\dnear \right)$ is called a \emph{descriptive proximity space}.   Unlike the Lodato Axiom (P2), the converse of the descriptive Lodato Axiom (dP2) also holds.

\begin{proposition}{\rm \cite{PetersTozzi2016arXivWiredFriend}}
Let $\left(X,\dnear\right)$ be a\\ descriptive proximity space, $A,B\subset X$.\\  Then $A\ \dnear\ B \Rightarrow A\ \dcap\ B\neq \emptyset$.
\end{proposition}

\subsection{Spatial and Descriptive Strong Proximities}
This section briefly introduces spatial strong proximity between nonempty sets and descriptive strong Lodato proximity.
Nonempty sets $A,B$ in a topological space $X$ equipped with the relation $\sn$, are \emph{strongly near} [\emph{strongly contacted}] (denoted $A\ \sn\ B$), provided the sets have at least one point in common.   The strong contact relation $\sn$ was introduced in~\cite{Peters2015JangjeonMSstrongProximity} and axiomatized in~\cite{PetersGuadagni2015stronglyNear},~\cite[\S 6 Appendix]{Guadagni2015thesis}.

\begin{definition}\cite{PetersGuadagni2015arXivFar}
We say that $A$ and $B$ are $\delta-$strongly far and we write $\mathop{\not{\delta}}\limits_{\mbox{\tiny$\doublevee$}}$ if and only if $A \not\delta B$ and there exists a subset $C$ of $X$ such that $A \not\delta X \setminus C$ and $C \not\delta B$, that is the Efremovi\v c property holds on $A$ and $B$.
\end{definition}

\begin{example}\cite{PetersGuadagni2015arXivFar}
In Figure~\ref{fig:stronglyFar}, let $X$ be a nonempty set endowed with the euclidean metric proximity $\delta_e$, $C,E\subset X, A\subset C, B\subset E$.  Clearly, $A\ \stackrel{\not{\text{\normalsize$\delta$}}_e}{\text{\tiny$\doublevee$}}\ B$ ($A$ is strongly far from $B$), since $A \not\delta_e B$ so that  $A \not\delta_e X \setminus C$ and $C \not\delta_e B$.  Also observe that the Efremovi\v{c} property holds on $A$ and $B$. \qquad \textcolor{black}{$\blacksquare$}
\end{example}

Let $X$ be a topological space, $A, B, C \subset X$ and $x \in X$.  The relation $\sn$ on the family of subsets $2^X$ is a \emph{strong proximity}, provided it satisfies the following axioms.

\begin{description}
\item[{\rm\bf (snN0)}] $\emptyset\ \sfar\ A, \forall A \subset X $, and \ $X\ \sn\ A, \forall A \subset X$.
\item[{\rm\bf (snN1)}] $A \sn B \Leftrightarrow B \sn A$.
\item[{\rm\bf (snN2)}] $A\ \sn\ B$ implies $A\ \cap\ B\neq \emptyset$. 
\item[{\rm\bf (snN3)}] If $\{B_i\}_{i \in I}$ is an arbitrary family of subsets of $X$ and  $A \sn B_{i^*}$ for some $i^* \in I \ $ such that $\Int(B_{i^*})\neq \emptyset$, then $  \ A \sn (\bigcup_{i \in I} B_i)$ 
\item[{\rm\bf (snN4)}]  $\mbox{int}A\ \cap\ \mbox{int} B \neq \emptyset \Rightarrow A\ \sn\ B$.  
\qquad \textcolor{blue}{$\blacksquare$}
\end{description}

\noindent When we write $A\ \sn\ B$, we read $A$ is \emph{strongly near} $B$ ($A$ \emph{strongly contacts} $B$).  The notation $A\ \sfar\ B$ reads $A$ is not strongly near $B$ ($A$ does not \emph{strongly contact} $B$). For each \emph{strong proximity} (\emph{strong contact}), we assume the following relations:
\begin{description}
\item[{\rm\bf (snN5)}] $x \in \Int (A) \Rightarrow x\ \sn\ A$ 
\item[{\rm\bf (snN6)}] $\{x\}\ \sn \{y\}\ \Leftrightarrow x=y$  \qquad \textcolor{blue}{$\blacksquare$} 
\end{description}

For strong proximity of the nonempty intersection of interiors, we have that $A \sn B \Leftrightarrow \Int A \cap \Int B \neq \emptyset$ or either $A$ or $B$ is equal to $X$, provided $A$ and $B$ are not singletons; if $A = \{x\}$, then $x \in \Int(B)$, and if $B$ too is a singleton, then $x=y$. It turns out that if $A \subset X$ is an open set, then each point that belongs to $A$ is strongly near $A$.  The bottom line is that strongly near sets always share points, which is another way of saying that sets with strong contact have nonempty intersection.   Let $\near$ denote a traditional proximity relation~\cite{Naimpally70withWarrack}.
$\mbox{}$\\
\vspace{3mm}  

\begin{figure}[!ht]
\centering
\includegraphics[width=65mm]{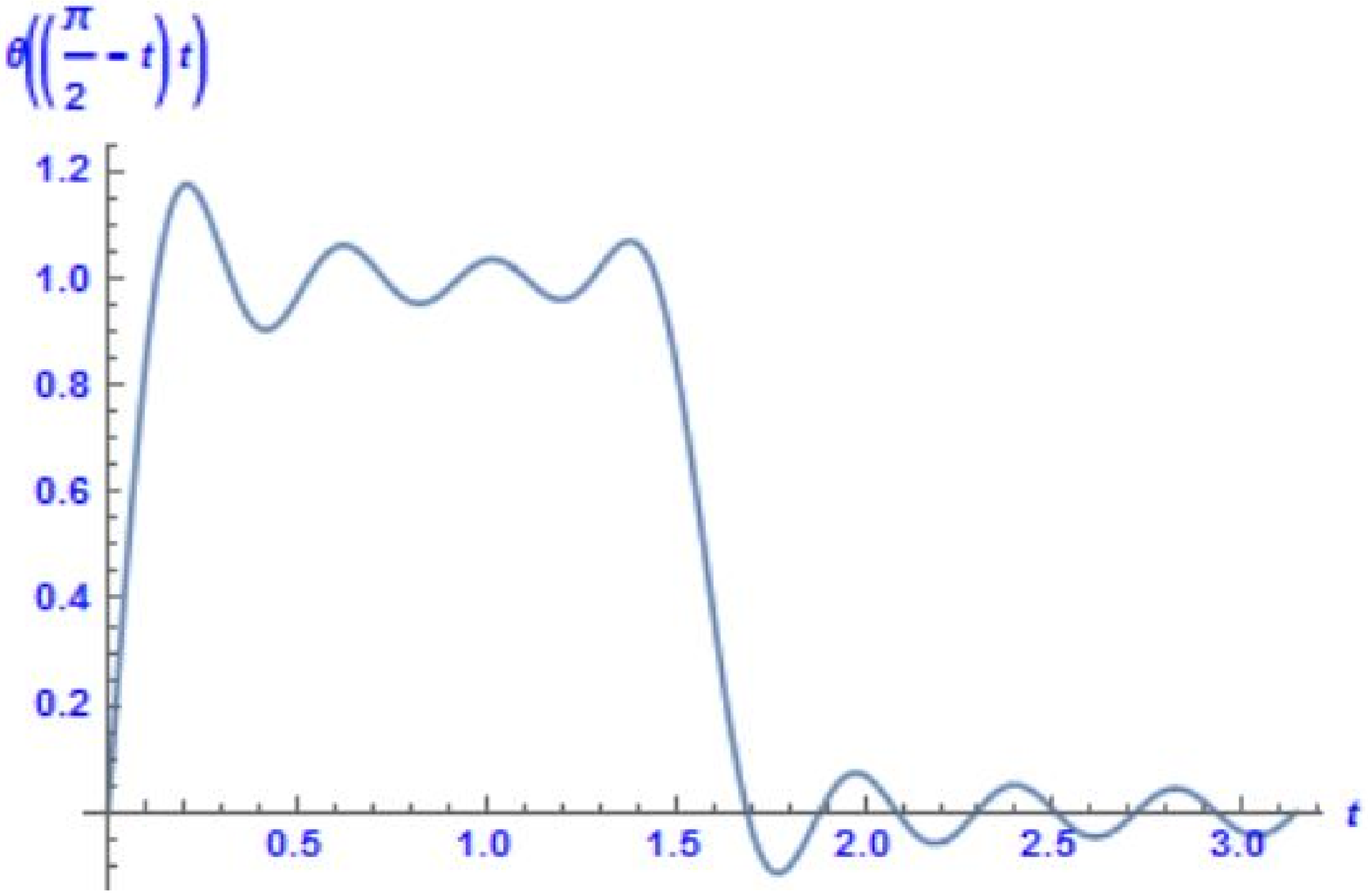}
\caption[]{Overlapping Sets:\\ $A = \left\{(x,0): 0.1\leq x\leq 3\right\},B = \left\{(x,\frac{2}{\pi}sin(x))+\cdots:0.1\leq x\leq 3\right\}$}
\label{fig:nearSetsStrong}
\end{figure}

Next, consider a proximal form of a Sz\'{a}z relator~\cite{Szaz1987}.  A \emph{proximal relator} $\mathscr{R}$ is a set of relations on a nonempty set $X$~\cite{Peters2016relator}.  The pair $\left(X,\mathscr{R}\right)$ is a proximal relator space.  The connection between $\sn$ and $\near$ is summarized in Prop.~\ref{thm:sn-implies-near}.

\begin{proposition}\label{thm:sn-implies-near}{\rm \cite{PetersTozzi2016arXivWiredFriend}}
Let $\left(X,\left\{\near,\dnear,\sn\right\}\right)$ be a proximal relator space, $A,B\subset X$.  Then 
\begin{compactenum}[1$^o$]
\item $A\ \sn\ B \Rightarrow A\ \near\ B$.
\item $A\ \sn\ B \Rightarrow A\ \dnear\ B$.
\end{compactenum}
\end{proposition}

\begin{example}
Let $X$ be a topological space endowed with the strong proximity $\sn$ and $A = \left\{(x,0): 0.1\leq x\leq 3\right\},B = \left\{(x,\frac{2}{\pi}sin(x))+\cdots:0.1\leq x\leq 3\right\}$.  The plot displays a curve connecting points in a Fourier sine series:
\[
\frac{2sin(t)}{\pi} + \frac{2sin(2t)}{\pi} + \frac{2sin(3t)}{3\pi} + \frac{2sin(5t)}{5\pi} + \frac{2sin(6t)}{3\pi} + \frac{2sin(7t)}{7\pi}+\cdots
\]
 In this case, $A,B$ represented by Fig.~\ref{fig:nearSetsStrong} are strongly near sets with many regions in common along the horizontal axis between 0 and 3.0.  
\qquad \textcolor{blue}{$\blacksquare$}
\end{example}

The descriptive strong proximity $\snd$ is the descriptive counterpart of $\sn$.
To obtain a \emph{descriptive strong Lodato proximity} (denoted by \emph{\bf dsn}), we swap out $\dnear$ in each of the descriptive Lodato axioms with the descriptive strong proximity $\snd$.  

Let $X$ be a topological space, $A, B, C \subset X$ and $x \in X$.  The relation $\snd$ on the family of subsets $2^X$ is a \emph{descriptive strong Lodato proximity}, provided it satisfies the following axioms.
\vspace{2mm}
\begin{description}
\item[{\rm\bf (dsnP0)}] $\emptyset\ {\sdfar}\ A, \forall A \subset X $, and \ $X\ \snd\ A, \forall A \subset X$.
\item[{\rm\bf (dsnP1)}] $A\ \snd\ B \Leftrightarrow B\ \snd\ A$.
\item[{\rm\bf (dsnP2)}] $A\ \snd\ B$ implies $A\ \dcap\ B\neq \emptyset$.  
\item[{\rm\bf (dsnP4)}] $\mbox{int}A\ \dcap\ \mbox{int} B \neq \emptyset \Rightarrow A\ \snd\ B$.  
\qquad \textcolor{blue}{$\blacksquare$}
\end{description}

\noindent When we write $A\ \snd\ B$, we read $A$ is \emph{descriptively strongly near} $B$.
For each \emph{descriptive strong proximity}, we assume the following relations:
\begin{description}
\item[{\rm\bf (dsnP5)}] $\Phi(x) \in \Phi(\Int (A)) \Rightarrow x\ \snd\ A$. 
\item[{\rm\bf (dsnP6)}] $\{x\}\ \snd\ \{y\} \Leftrightarrow \Phi(x) = \Phi(y)$.  
\qquad \textcolor{blue}{$\blacksquare$} 
\end{description}

So, for example, if we take the strong proximity related to non-empty intersection of interiors, we have that $A\ \snd\ B \Leftrightarrow \Int A\ \dcap\ \Int B \neq \emptyset$ or either $A$ or $B$ is equal to $X$, provided $A$ and $B$ are not singletons; if $A = \{x\}$, then $\Phi(x) \in \Phi(\Int(B))$, and if $B$ is also a singleton, then $\Phi(x)=\Phi(y)$. 

The connections between $\snd,\dnear$ are summarized in Prop.~\ref{thm:sn-implies-dnear}.  
 
\begin{proposition}\label{thm:sn-implies-dnear}{\rm \cite{PetersTozzi2016arXivWiredFriend}}
Let $\left(X,\left\{\sn,\dnear,\snd\right\}\right)$ be a proximal relator space, $A,B\subset X$.  Then 
\begin{compactenum}[1$^o$]
\item For $A,B$ not equal to singletons, $A\ \snd\ B \Rightarrow \Int A\ \dcap\ \Int B\neq \emptyset \Rightarrow \Int A\ \dnear\ \Int B$.
\item $A\ \sn \ B \Rightarrow (\Int A\ \dcap\ \Int B)\ \snd \ B$.
\item $A\ \snd\ B \Rightarrow A\ \dnear\ B$.
\end{compactenum}
\end{proposition}

\section{Region-Based Physical Geometry}

In this form of geometry, the analogue of an abstract point is a region of space with minimal but not zero area.  Let $X$ be a set of physical regions.  

\setlength{\intextsep}{0pt}

\begin{wrapfigure}[12]{R}{0.35\textwidth}
\begin{minipage}{5.2 cm}
\centering
\includegraphics[width=40mm]{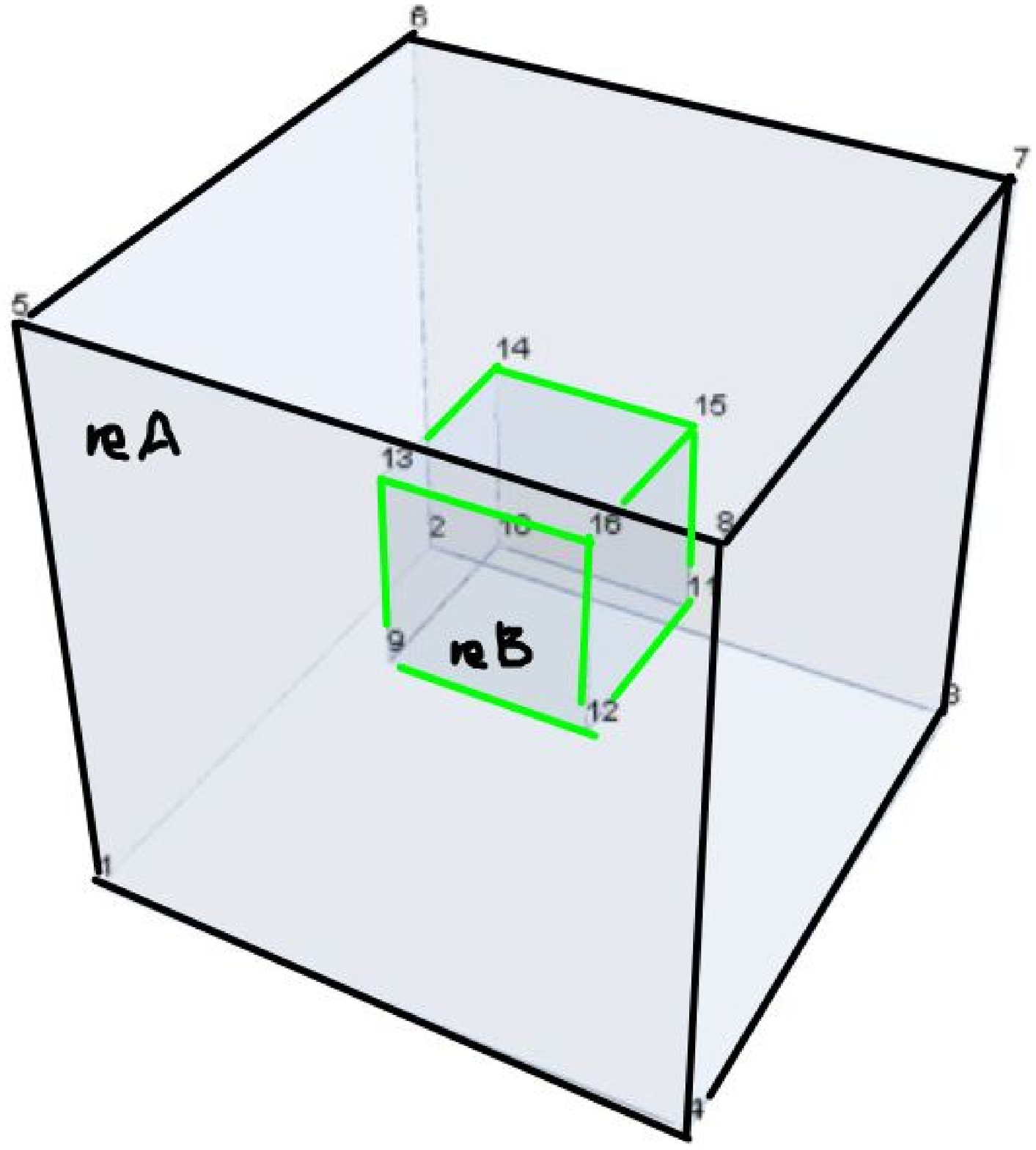}
\caption[]{$\mbox{}$\\ 3D Hole}
\label{fig:concentricRegions}
\end{minipage}
\end{wrapfigure}

A \emph{region} $A$ in a space $X$ is a collection of subregions.  A region is denoted either by $\re A$ or simply by $A$.  The coordinates of the center of mass of a region is the location of the region.  A \emph{boundary region} of $\re A$ is the set of regions that have an subregions in common with $\re A$.  A region is \emph{open}, provided the region does not include its boundary region. Let $\Int A$ denote the interior of region $A$, defined by 
$\Int A = \bigcup\left\{B\subset X: B\ \mbox{is open and}\ B\subset A\right\}$.  

A \emph{hole} is a closed region with an empty interior.  For example, the 3D cube $\re B$ in the interior of cube $\re A$ in Fig.~\ref{fig:concentricRegions} is an example of a hole, since $\Int(\re B) = \emptyset$ (the $\re B$ has only edges, {\em i.e.}, each of its faces has an empty interior bounded by edges).

A subregion (subset) of a region is denoted by an upper case letter, {\em e.g.}, $E\subset \re B$ or $E\subset B$.  A member (small subregion) of a region is denoted by a lower case letter, {\em e.g.}, $x\in \re A$ or $p\in A$.

Let $X$ be endowed with the Lodato proximity $\near$.  The Cech distance $D(A,B)$ between regions $A,B$ is defined by\\
$
D(A,B) = \mbox{inf}\left\{d(a,b): a\in A, b\in B\right\}.  
$
Then define $A\ \near\ B$, if and only if $D(A,B)= 0$~\cite{Naimpally70withWarrack}. The \emph{closure} of set $A\subset X$ (denoted $\cl A$) is the set $\cl A = \left\{x\in X: {x}\ \near\ A\right\}$.  Let $a\in A$.  The \emph{neighbourhood of} $a$ is the set $\nhbd_{_{\varepsilon}}(a)= \left\{x\in X: d(x,a) < \varepsilon \right\}$.  
The neighbourhood $\nhbd_{_{\varepsilon}}(a)$ is an open set such that each $x\in \nhbd_{_{\varepsilon}}(a)$ is \emph{sufficiently close} to $a$, {\em i.e.}, $d(x,a) < \varepsilon$.  
In general, a nonempty set $A$ is an \emph{open set}, if and only if all points $x\in X$ sufficiently close to $A$ belong to $A$~\cite{Bourbaki1966}. The \emph{boundary} of a set $A$ is the set $\bdy A = \left\{x\in X: \nhbd_{_{\varepsilon}}(x)\cap \left(A\ \cap\ X\setminus A\right)\right\}$.

\begin{proposition}{\rm \cite{Krantz2009}}\label{Krantz2009}
Let $X$ be a metric topological space, $\re A\subset X$.  $\cl A = A\cup \bdy A$.
\end{proposition}

For $A\in 2^X$, $\Phi(A)$ is a description of region $A$. A special bornology is given by the $\Phi-$bounded sets. We say that a subset $A$ is \emph{$\Phi-$bounded} (denoted by $\bnd A$), provided there exists $\varepsilon \in \mathbb{R}^+ $ such that $A \subseteq \{ B \in X : d(\Phi(B), \Phi(E))< \varepsilon\}$, where $d$ is a pseudo-metric on $\mathbb{R}^n$ and $E$ is a selected region in $X$.  The closure of a $\bnd A$ is defined by
\[
\cl\left(\bnd A\right) = \left\{x\in X: x\ \near\ \bnd A\right\},
\]
{\em i.e.}, $x\in cl\left(\bnd A\right)$, provided $\Phi(x) = \Phi(y)\ \mbox{for some}\ y\in \bnd A$.

\begin{lemma}\label{lemma:bnd}
Let $X$ be a metric topological space, $\bnd A\subset X$.  $\cl (\bnd A) = \bnd A\cup \bdy (\bnd A)$.
\end{lemma}
\begin{proof}$\mbox{}$\\
$\subseteq$: Let $x\not\in \bnd A\cup \bdy (\bnd A)$.  Since $x$ is not in $\bdy (\bnd A)$, there is a neighbourhood $\nhbd(x)\cap \bnd A = \emptyset$ or $\nhbd(x)\cap \left(X\setminus \bnd A = \emptyset\right)$.   We know that $x\not\in \bnd A$.  Consequently, $\nhbd(x)\subseteq X\setminus \bnd A$, {\em i.e.}, $\nhbd(x)$ is a subset of the complement of $\bnd A$.  Then  
\[
\nhbd(x)\ \mbox{is not in}\ \bnd A\cup \bdy (\bnd A),\ \mbox{for some}\ \varepsilon > 0.
\]
Also, every $y$ in $\nhbd(x)$ is not in $\bnd A\cup \bdy (\bnd A)$.  Thus, the complement of $\bnd A\cup \bdy (\bnd A)$ is open and $\bnd A\cup \bdy (\bnd A)$ is closed.  Hence, $\bnd A\cup \bdy (\bnd A)\supseteq \cl (\bnd A)$.\\
$\supseteq$: Assume $x\not\in \cl (\bnd A)$.  Since the complement of $\cl (\bnd A)$ (denoted by $X\setminus \cl (\bnd A)$) is open, $\nhbd(x)\subseteq X\setminus \cl (\bnd A)$.  Then, $\nhbd(x)\cap \bnd A = \emptyset$.  Hence, $\bnd A\supseteq \cl (\bnd A)$.
\end{proof}

In a metric topological space $X$, a set $A\subset X$ is closed, provided its complement is open~\cite{Krantz2009}.

\begin{proposition}\label{prop:clbndClosed}
Let $X$ be a metric topological space, $\bnd A\subset X$.  $\cl (\bnd A)$ is closed.
\end{proposition}
\begin{proof}
From the proof of Lemma~\ref{lemma:bnd}, $\cl (\bnd A)$ is closed, since its complement is open.
\end{proof}

\begin{theorem}\label{lemma:closureAequalsA}
Let $X$ be a metric topological space, $\bnd A\subset X$.  $\cl (\bnd A) = \bnd A$.
\end{theorem}
\begin{proof}$\mbox{}$\\
$\subseteq$: Let $x\not\in \bnd A$.  Since $x$ is not in $\bnd A$, there is a neighbourhood $\nhbd(x)\cap \bnd A = \emptyset$ or $\nhbd(x)\cap \left(X\setminus \bnd A\right) = \emptyset$.   Consequently, $\nhbd(x)\subseteq \left(X\setminus \bnd A\right)$, {\em i.e.}, $\nhbd(x)$ is a subset of the complement of $\bnd A$.  Also, every $y$ in $\nhbd(x)$ is not in $\bnd A$.  Thus, the complement of $\bnd A$ is open and $\bnd A$ is closed.  Hence, $\bnd A\supseteq \cl (\bnd A)$.\\
$\supseteq$: Assume $x\not\in \cl (\bnd A)$.  Since the complement of $\cl (\bnd A)$ is open, $\nhbd(x)\subseteq X\setminus \cl (\bnd A)$ and $\nhbd(x)\subset X\setminus \bnd A$.  Then, $\nhbd(x)\cap \left(\bnd A\cup \bdy (\bnd A)\right) = \emptyset$.  Consequently, $x\not\in \left(\bnd A\cup \bdy (\bnd A)\right)$.  Hence, $\left(\bnd A\cup \bdy (\bnd A)\right)\subseteq \cl (\bnd A)$ and $\bnd A\subseteq \cl (\bnd A)$.
\end{proof}

\setlength{\intextsep}{0pt}

\begin{wrapfigure}[19]{L}{0.55\textwidth}
\begin{minipage}{7.2 cm}
\centering
\includegraphics[width=70mm]{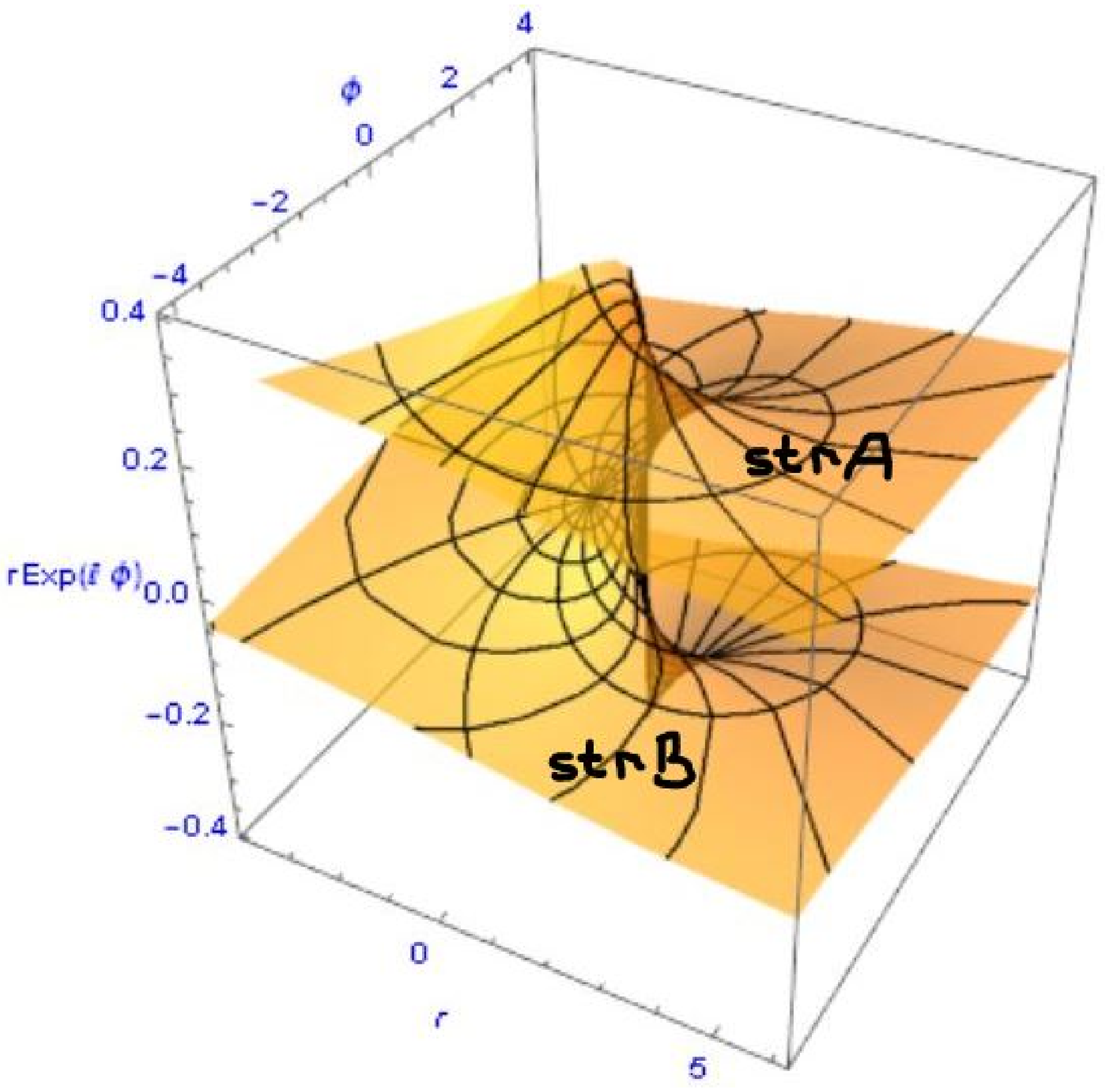}
\caption[]{$\mbox{}$\\ Nome strings}
\label{fig:strings0}
\end{minipage}
\end{wrapfigure} 

The intersection of regions $A,B\in X$ (denoted by $A\cap B$) is the set of all subregions common to $A$ and $B$.  The diameter of a region $A$ is the maximum distance between a pair of members of $A$.  
Let $p_1,\dots,p_{i-1},p_i,p_{i+1},\dots,p_k$ be the locations of the centers of mass of $n$ regions.  Let $\overline{p_i,p_{i+1}}$ be a line segment between $p_i,p_{i+1}$.  A \emph{path} is defined by a sequence of line segments such that
$
\left\{\overline{p_1,p_{2}},\dots,\overline{p_{i-1},p_i},\dots,\overline{p_{k-1},p_k}\right\}.   
$
Regions $\re A, \re B$ are \emph{connected}, provided $\re A\ \sn\ \re B$, {\em i.e.}, $\re A, \re B$ have members in common.  A \emph{path} between regions is a sequence of pairwise-connected regions.  That is, $\re A, \re B$ are adjacent members of a path, provided $\re A\ \sn\ \re B$. Regions $\re A, \re B$ are \emph{path-connected}, provided there is a path between $\re A$ and $\re B$.  Regions $\re A, \re B$ are contiguous, provided $\re B\cap \re A\neq \emptyset$, {\em i.e.}, $\re A, \re B$ have a common subregion.  A \emph{closed region} is identified with its closure.  The closure of a region $A$ (denoted by $\cl A$) is defined by
\[
\cl A = \left\{B\in X: B\cap A\neq \emptyset \right\}\ \mbox{(Closure of a region)}.
\]
Let $\emptyset$ denote the empty set, {\em i.e.}, a set containing no regions.

A \emph{surface} is a compact set of connected regions with a boundary.   A surface in physical geometry contrasts with a surface in abstract geometry.  For example, an Andrews surface is a compact two-manifold without a boundary~\cite{Andrews1988AMMsurface}.  A space is \emph{compact}, provided every open cover has a finite subcover~\cite[\S 17.1]{Willard1970}.  A \emph{cover} of a set $A$ is a collection of open sets $E$ whose union is a superset of $A$, {\em i.e.}, $A \subseteq \bigcup E$.  Let $X$ be a set of regions.  The space $X$ is connected, provided any two regions of $X$ can be joined by a region such as a line segment.
A \emph{manifold} is a topological space that is locally Euclidean~\cite{Rowland2016Manifold,Lee2013smoothManifold}.

\begin{proposition}
Every nonempty region is connected.
\end{proposition}
\begin{proof}
Let $A$ be a nonempty region.  Let $A$ contains subregions adjacent subregions such that $A = B\cup C$.  Then subregions $B$ and $C$ join together to form $A$.  Hence, the result follows. 
\end{proof}

\begin{proposition}
Every region in a physical geometry is compact.
\end{proposition}
\begin{proof}
Let $A$ be a physical region.  Let $X$ be a collection of connected open regions, {\em i.e.}, regions $E\subset X$ without boundaries.  Let $A$ be a proper subset of $X$.  Then $A \subseteq \mathop{\bigcup}\limits_{E\subset X} E$.  Let $C = \mathop{\bigcup}\limits_{E\subset X} E$.  Let $E'= \cl E$ for some $E'\in 2^X$.  Then $A \subseteq \mathop{\bigcup}\limits_{E\subset X} E\cup E'$.  Hence, each cover of $A$ has a subcover.  
\end{proof}

\begin{example} {\bf Sample Surfaces and Intersecting lines}.\\
\includegraphics[width=20mm]{2worldsheet}\ (Cylindrical surface),\ 
\includegraphics[width=20mm]{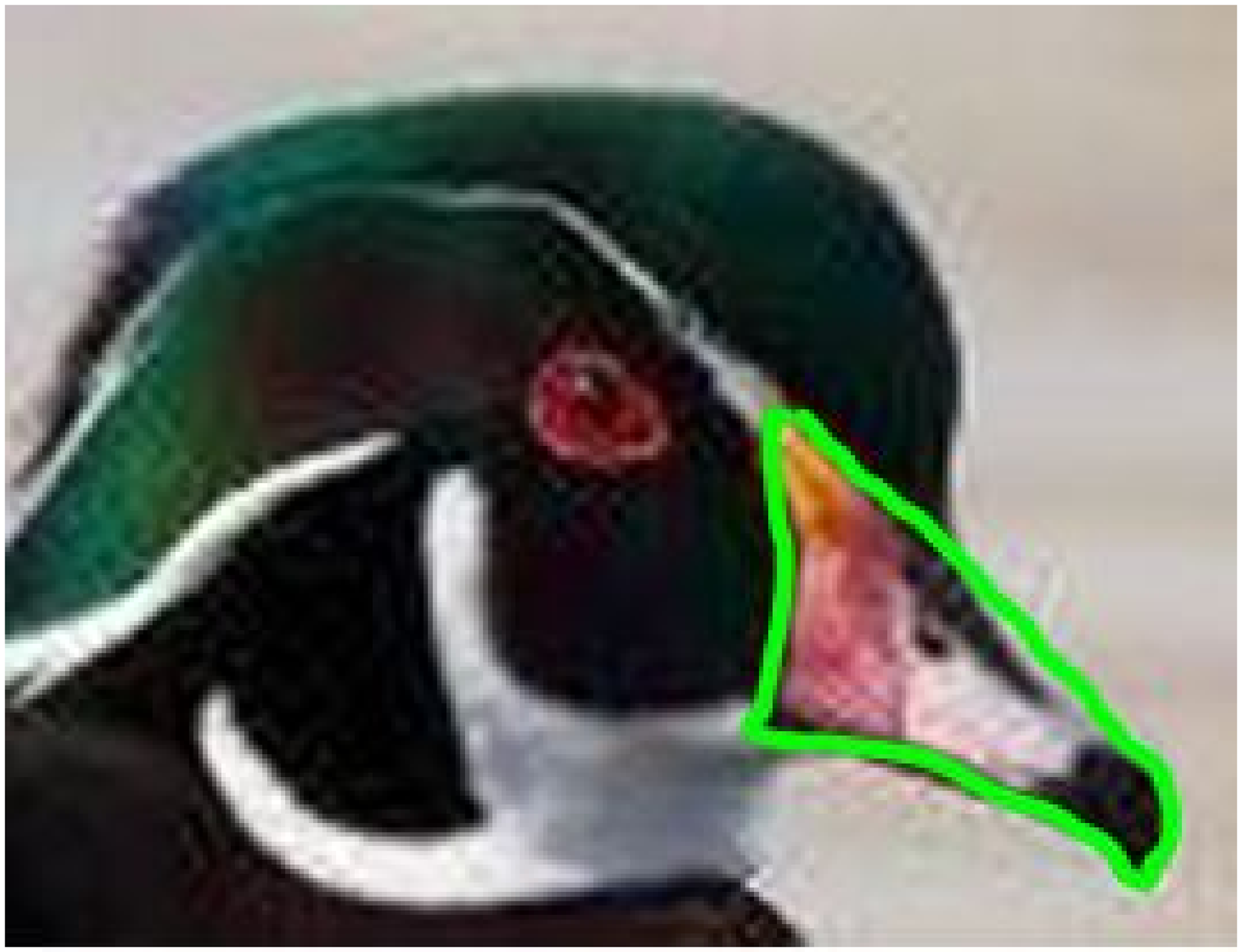}\ (Curved surface).  The bounded region of the wood duck's beak is a bounded surface delineated by the bounding green line.\\
\includegraphics[width=20mm]{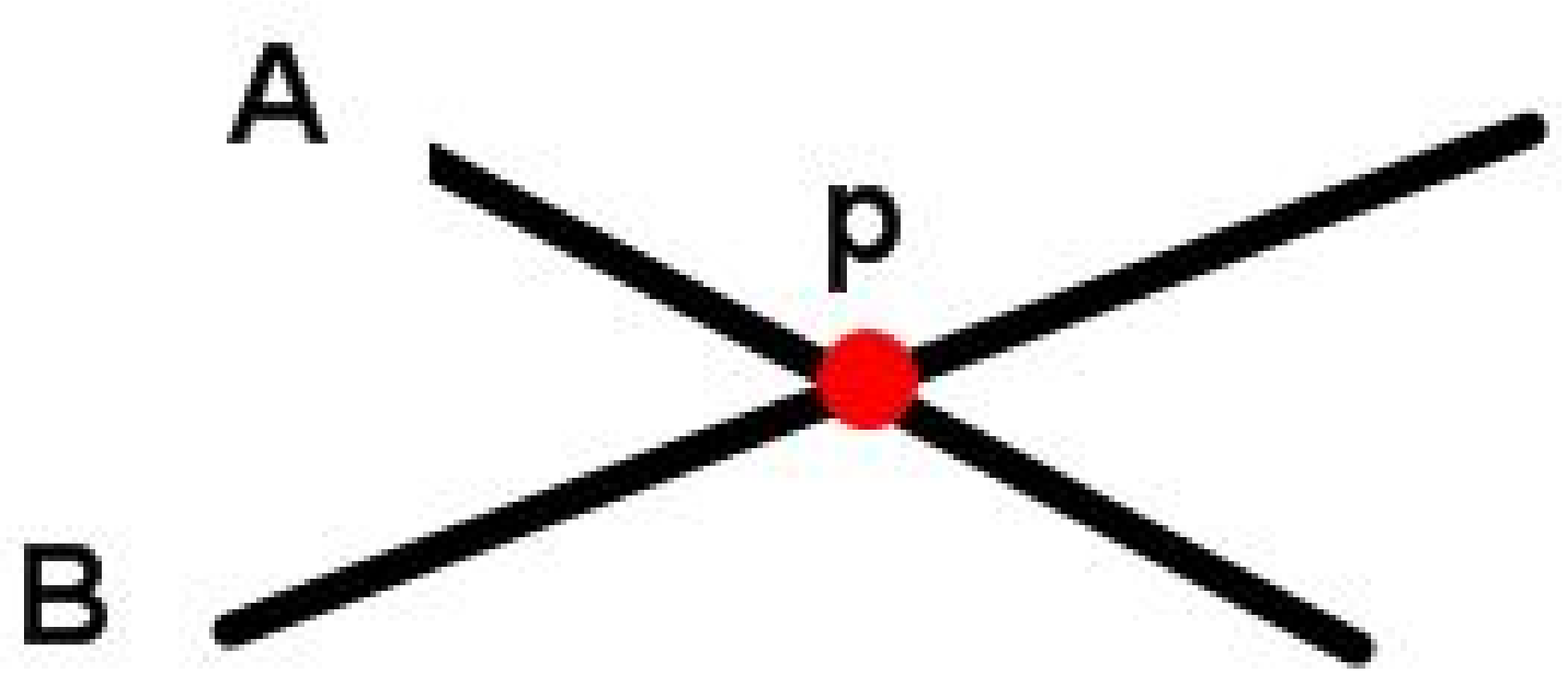}\ (Intersecting lines).  Line $A$ intersects line $B$ at region $p$ such that $A\ \sn\ B$, {\em i.e.}, $A$ and $B$ have a region in common.
\qquad \textcolor{blue}{$\blacksquare$}
\end{example}  

\section{Physical Geometry Axioms}
This section introduces the axioms for a physical geometry.\\

\noindent{\bf Axioms}
\begin{compactenum}[{\bf PG}.1]
\item\label{ax:region} A line $L$ drawn from any region to any region is straight if and only if $L$ has a constant slope.
\item\label{ax:area} Every region has non-zero area.
\item\label{ax:dim} Every region is finite.
\item $\Int(\re A)$ is a subregion of the closed region $\re A$ in physical space if and only $\Int(\re A) \subset \cl A$.
\item $\Int(\re A) = \emptyset$ if and only if $\re A$ contains no subregions.
\item $\re A$ is a closed region if and only if $\re A = \cl(\re A)$.
\item A circle is a closed polygon so that all regions along the boundary are equidistant from a central region. 
\item Regions are congruent, provided the regions have equal areas. 
\item A surface cuts a surface in a line. 
\item A line cuts a line in a region.
\item\label{ax:parallel} Lines $L,L'$ are locally parallel, provided $L$ has a line segment that does not cut $L'$.
\qquad \textcolor{blue}{$\blacksquare$}
\end{compactenum}
$\mbox{}$\\
\vspace{3mm}
Axiom~\ref{ax:region} is a variation of Euclid's Postulate 1.:
\begin{description}
\item[{\bf Postulate 1.}] To draw a straight line from any point to any point~\cite[Book I]{EuclidElements300BC}. \qquad \textcolor{blue}{$\blacksquare$}
\end{description}
Instead of the points at the ends of a straight line in Euclid's Posulate 1, the ends of a physical straight line are regions.  The diameter of the endpoints will equal the width of the line.  A \emph{line} is defined by a pair of path-connected regions $\re A, \re B$ (denoted by $\overline{\re A\re B}$), provided each pair of regions in the path have a common edge. A straight line in descriptive physical geometry has the look and feel of a straight line in Euclidean geometry.  By definition, a straight line is a region of space.  Every subregion of a line has the same area. From Axiom~\ref{ax:area}, a straight line has non-zero area, \emph{i.e.}, a straight line has non-zero length and width. 

Let $\str A, \str B$ be a pair strings.  A simplicial complex can be obtained (\emph{sewn or stitched together}) from $\str A, \str B$ by adding edges to the strings so that new edges intersect.  Examples of sewing operations are given in M.B. Green, J.H. Schwarz and E. Witten~\cite[\S 1.1, p. 3]{GreenScharzWitten2012CUPsuperstrings} and G. Moore and N. Seiberg~\cite[p. 196]{MooreSeiberg1989CMPconformalFields}.

\begin{figure}[!ht]
\centering
\includegraphics[width=30mm]{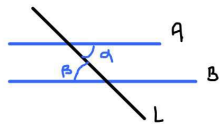}
\caption[]{Local Parallel Lines}
\label{fig:parallelLines}
\end{figure}

\begin{definition}
A pair of lines$A,B$ are \emph{locally parallel} (denoted by $A \| B$), provided the lines can be cut by a straight line
so that the interior angles are right angles.  Let $A,B$ be a pair of straight lines, $L$ a line that cuts $A$ and $B$, $\Int \angle LA, \Int \angle LB$ the interior angles.  Then $A \| B$, provided
\[
L\ \sn\ A, L\ \sn\ B,\ \mbox{and}\ \Int \angle LA = \Int \angle LB = 90^o.\mbox{\qquad \textcolor{blue}{$\blacksquare$}}
\]
\end{definition} 

\begin{figure}[!ht]
\begin{center}
\begin{pspicture}
 (0.0,0.5)(2.5,3.5)
\psframe[linecolor=black](-0.8,0.0)(4.5,3.5)
\pscircle[linecolor=black,linestyle=dotted,dotsep=0.05,fillstyle=solid,fillcolor=yellow](0.58,1.55){1.28}
\psline[linewidth = 2pt,linecolor = blue](0.58,0.55)(0.58,2.55)
\psline[linestyle=dotted,dotsep=0.05,linewidth = 2pt,linecolor = blue](0.58,2.55)(0.58,3.50)
\pscircle[linecolor=black,linestyle=dotted,dotsep=0.05,fillstyle=solid,fillcolor=yellow](3.38,1.85){1.00}
\psline[linewidth = 2pt,linecolor = blue](3.38,1.05)(3.38,2.55)
\psline[linestyle=dotted,dotsep=0.05,linewidth = 2pt,linecolor = blue](3.38,2.55)(3.38,3.50)
\rput(-0.5,3.2){\large  $\boldsymbol{X}$}
\rput(0.88,1.35){\footnotesize $\boldsymbol{A}$}
\rput(0.08,2.55){\large $\boldsymbol{C}$}
\rput(3.68,1.65){\footnotesize $\boldsymbol{B}$}
\rput(2.80,2.55){\large  $\boldsymbol{E}$}
 \end{pspicture}
\end{center}
\caption{\footnotesize Strongly Far Parallel Lines}
\label{fig:stronglyFar2}
\end{figure}
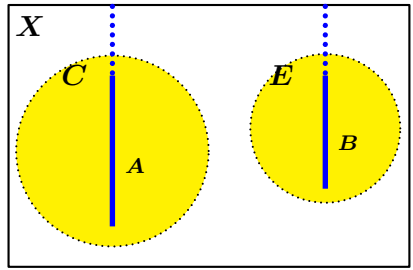
$\mbox{}$\\
\vspace{3mm}

\section{Strongly Far Parallel Lines}

Axiom~\ref{ax:parallel} contrasts with Euclid's Fifth Postulate.\\
\vspace{3mm} 
\begin{description}
\item[{\bf Euclid's Fifth Postulate}] That, if a straight line falling on two straight lines make the interior
angles on the same side less than two right angles, if produced indefinitely, meet on that side on which are the
angles less than two right angles.
\end{description}
$\mbox{}$\\
\vspace{3mm}

From Axiom~\ref{ax:parallel}, a pair of lines are locally parallel, provided $L$ and $L'$ have no line segments that intersect.   Similarly, surfaces $S,S'$ are locally parallel, provided $S$ has no subregions that cut $S'$.  Let $X$ be a set of regions equipped with the relator that contains the Lodator proximity $\near$ and the strong proximity $\sn$.  This gives us the proximity space $\left(X,\left\{\near,\sn\right\}\right)$. The focus here is regions of physical space that are far apart and strongly far apart.

\begin{definition}
We say that $A$ and $B$ are $\delta-$strongly far and we write $\mathop{\not{\delta}}\limits_{\mbox{\tiny$\doublevee$}}$ if and only if $A \not\delta B$ and there exists a subset $C$ of $X$ such that $A \not\delta X \setminus C$ and $C \not\delta B$, that is the Efremovi\v c property holds on $A$ and $B$.
\qquad \textcolor{black}{$\blacksquare$}
\end{definition}

Observe that $A \not \delta B$ does not imply $A \mathop{\not{\delta}}\limits_{\mbox{\tiny$\doublevee$}} B$. In fact, this is the case when the proximity $\delta$ is not an \emph{EF-proximity}.

This gives rise to a view of parallelism that does not depend on the parallel lines being straight and does not depend on the lines being cut by a straight line whose interior angles are both right angles.
$\mbox{}$\\
\vspace{3mm}

\begin{description}
\item[{\bf PG.12 Proximal Parallel Axiom}] If lines $A, B$ are extended indefinitely and $A\ \sfar\ B$, then $A$ and $B$ are parallel.  That is, no part of line $A$ overlaps line $B$.
\qquad \textcolor{black}{$\blacksquare$}
\end{description}
$\mbox{}$\\
\vspace{3mm}

By considering the gap between two sets in a metric space ( $d(A,B) = \inf \{d(a,b): a \in A, b \in B\}$ or $\infty$ if $A$ or $B$ is empty ), Efremovi\v c introduced a stronger proximity called \textit{Efremovi\v c proximity} or \textit{EF-proximity}.  

\begin{definition}
An \emph{EF-proximity} is a relation on $\mathscr{P}(X)$ which satisfies $P0)$ through $P3)$ and in addition 
\[A 
\not\delta B \Rightarrow \exists E \subset X \hbox{ such that } A \not\delta E \hbox{ and } X\setminus E \not\delta B\ \mbox{({\bf EF-property})}.
\]
\end{definition}

\begin{example}
In Fig.~\ref{fig:stronglyFar2}, let $X$ be a nonempty set of regions endowed with the Euclidean metric proximity $\delta_e$, $C,E\subset X$, line $A\subset C$, line $B\subset E$.  Clearly, $A\ \stackrel{\not{\text{\normalsize$\delta$}}_e}{\text{\tiny$\doublevee$}}\ B$ (line $A$ is strongly far from line $B$, since there is no sub-region of line $A$ that overlaps any part of line $B$), since $A \not\delta_e B$ so that  $A \not\delta_e X \setminus C$ and $C \not\delta_e B$.  Also observe that the Efremovi\v{c} property holds on $A$ and $B$.  This also means that $d(A,B) = 0$, {\em i.e.}, the gap between lines $A$ and $B$ is always non-zero.  Hence, from the {\bf PG.12} Proximal Parallel Axiom, $A$ and $B$ are parallel.  \qquad \textcolor{black}{$\blacksquare$}
\end{example}

\begin{lemma}\label{lem:parallelRegions}
Let $\re A,\re B$ be nonempty regions.  If $\re A$ and $\re B$ are expanded indefinitely and $\re A\ \sfar\ \re B$, then $\re A$ and $\re B$ are parallel.
\end{lemma}
\begin{proof}
Partition regions $\re A$ and $\re B$ into lines that are open sets, {\em i.e.}, each line does not include into border points. Since region is extended indefinitely, each regional line is extended indefinitely.  Let line $L_A\in \re A$, line $L_B\in \re B$.
$L_A\ \sfar\ L_B$, since $\re A\ \sfar\ \re B$.  Consequently, from the {\bf PG.12} Proximal Parallel Axiom, $L_A\ \|\ L_B$.  Since this holds true for all pairs of regional lines, $\re A\ \|\ \re B$.
\end{proof}

\begin{theorem}
Let $\str A,\str B$ be nonempty strings.  If $\str A$ and $\str B$ are expanded indefinitely and $\str A\ \sfar\ \str B$, then $\str A$ and $\str B$ are parallel.
\end{theorem}
\begin{proof}
Strings $\str A,\str B$ are regions that have been expanded indefinitely and $\str A\ \sfar\ \str B$.  Hence, by Lemma~\ref{lem:parallelRegions}, $\str A\ \|\ \str B$.
\end{proof}

\begin{theorem}
Let $\wsh A,\wsh B$ be nonempty worldsheets.  If $\wsh A$ and $\wsh B$ are expanded indefinitely and $\wsh A\ \sfar\ \wsh B$, then $\wsh A$ and $\wsh B$ are parallel.
\end{theorem}
\begin{proof}
Worldsheets $\wsh A,\wsh B$ are regions that have been expanded indefinitely and $\wsh A\ \sfar\ \wsh B$.  Hence, by Lemma~\ref{lem:parallelRegions}, $\wsh A\ \|\ \wsh B$.
\end{proof}

\section{Sewing Regions Together}
This section introduces a sew operation.  The basic idea is to introduce an edge $L$ between a pair of parallel regions $\re A, \re B$ so that $\re A, \re B$ are connected and a simplicial complex is constructed.  That is, $\re A\ \sfar\ \re B$ (the regions do not overlap, even if they are extended indefinitely).  Let $L$ have vertices $p,q$ so that $L = \overline{p,q}$ and let $\re A\ \sn p$ (region $\re A$ and $p$ overlap).  Further, let $\re A\ \sn q$ (region $\re A$ and $q$ overlap).  Then
\[
\re A\ \sn\ L\ \&\ \re B\ \sn\ L,
\]
forming a simplicial complex.  In general, non-overlapping regions are sewn together by joining the regions by one or more edges.  Let $k\in \mathbb{N}$ be a natural number. Let $2^{\mathbb{R}^2}$ be a collection of plane regions. From this, the planar mapping $sew:2^{\mathbb{R}^2}\times2^{\mathbb{R}^2}\times\mathbb{N}\longrightarrow 2^{\mathbb{R}^2}$ is defined by
\[
sew(\re A,\re B,k) = \left\{\re A, \re B\right\}\cup \mathop{\bigcup}\limits_{i=1}^{k}\left\{a\in\re A,b\in\re B,\left\{p,q\right\}\in L: a\ \sn\ p\ \mbox{$\&$}\ b\ \sn\ q\right\}.
\]

By sewing a pair of plane regions together with $sew(\re A,\re B, 1)$, $\re A\ \sn\ p, \re B\ \sn\ q$, we mean an edge is added between a subregion $\re a\in \re A$ and vertex $p$ in line $\overline{p,q}$ and a subregion $\re b\in \re B$ and vertex $q$ in line $\overline{p,q}$.\\
\vspace{3mm}

\begin{figure}[!ht]
\centering
\includegraphics[width=50mm]{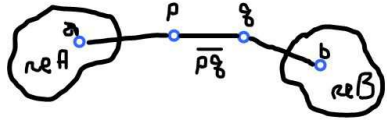}
\caption[]{Sewing Disconnected Regions Together}
\label{fig:sewingRegions}
\end{figure}

\begin{example}{\bf Stitching Regions Together}.\\
Let regions $\re A, \re B$ contain subregions $\re\ a, \re\ b$, respectively, as shown in Fig.~\ref{fig:sewingRegions}. Assume that $\re A, \re B$ are disconnected, {\em i.e.}, there is no path between $\re A, \re B$. The $sew(\re A,\re B, k=1)$ operation transforms $\re A, \re B$ into connected regions by introducing at least one edge $\overline{p,q}$ so that vertex $p$ is connected to a subregion $\re\ a$ and vertex $q$ is connected to a subregion $\re\ b$ as shown in Fig.~\ref{fig:sewingRegions}.
\qquad \textcolor{black}{$\blacksquare$}
\end{example}

\begin{lemma}\label{lem:sewingRegions}
$sew(\re A,\re B)$ constructs a simplicial complex.
\end{lemma}
\begin{proof}
$sew(\re A,\re B,k)$ is defined by a set of edges $\left\{\overline{p,q}\right\}$ such that each the vertices of edge are connected to subregions in $\re A,\re B$.  Without loss of generality, let vertex $p$ in edge $\overline{p,q}$ be connected to subregion $\re\ a\in \re A$ and let vertex $q$ in edge $\overline{p,q}$ be connected to subregion $\re\ b\in \re B$.  By definition, $\re A\ \sn\ p\ \&\ \re A\ \sn\ q$.  Hence, $\re A\cup\overline{p,q}\cup\re B$ is a simplicial complex.  If regions $\re A,\re B$ are connected to $k>0$ edges as a result of $sew(\re A,\re B,k)$, then $sew(\re A,\re B,k)$ constructs a simplicial complex with multiple connecting edges.  This gives the desired result. 
\end{proof}

\begin{lemma}\label{lem:sewingOp}
Let $\re A, \re B$ be simplicial complexes.  If $sew(\re A,\re B,k)$, then the resulting region is a simplicial complex.
\end{lemma}
\begin{proof}
From Lemma~\ref{lem:sewingRegions}, $sew(\re A,\re B,k)$ is a simplicial complex.  Let $sew(\re A,\re B,k)$ contain a line $\overline{p,q}$ so that $\re A\ \sn\ p\ \&\ \re A\ \sn\ q$, {\em i.e.}, vertex $p$ is connected to a subregion $\re\ a\in \re A$ and vertex $q$ is connected to a subregion $\re\ b\in \re B$.  This gives the desired result. 
\end{proof}

Let regions $\re A, \re B$ have vertices $p,q$, respectively. The application of $sew(\re A,\re B)$ on parallel regions $\re A, \re B$ results in a closed polygon. 

\begin{figure}[!ht]
\centering
\subfigure[Disconnected regions]
 {\label{fig:reP}\includegraphics[width=25mm]{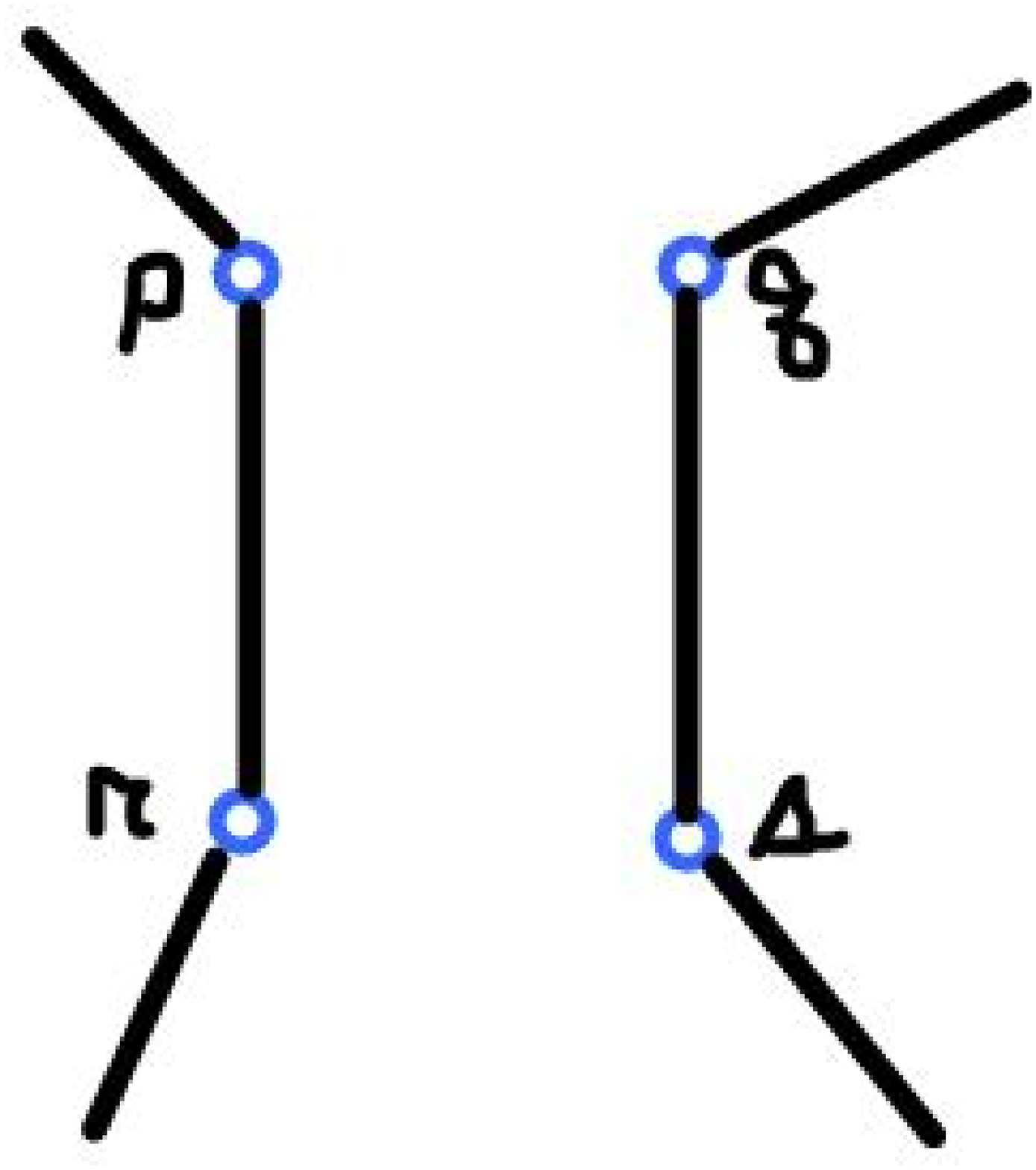}}\hfil
\subfigure[Connected subregions $\re\ p,\re\ q$]
 {\label{fig:rePQ}\includegraphics[width=25mm]{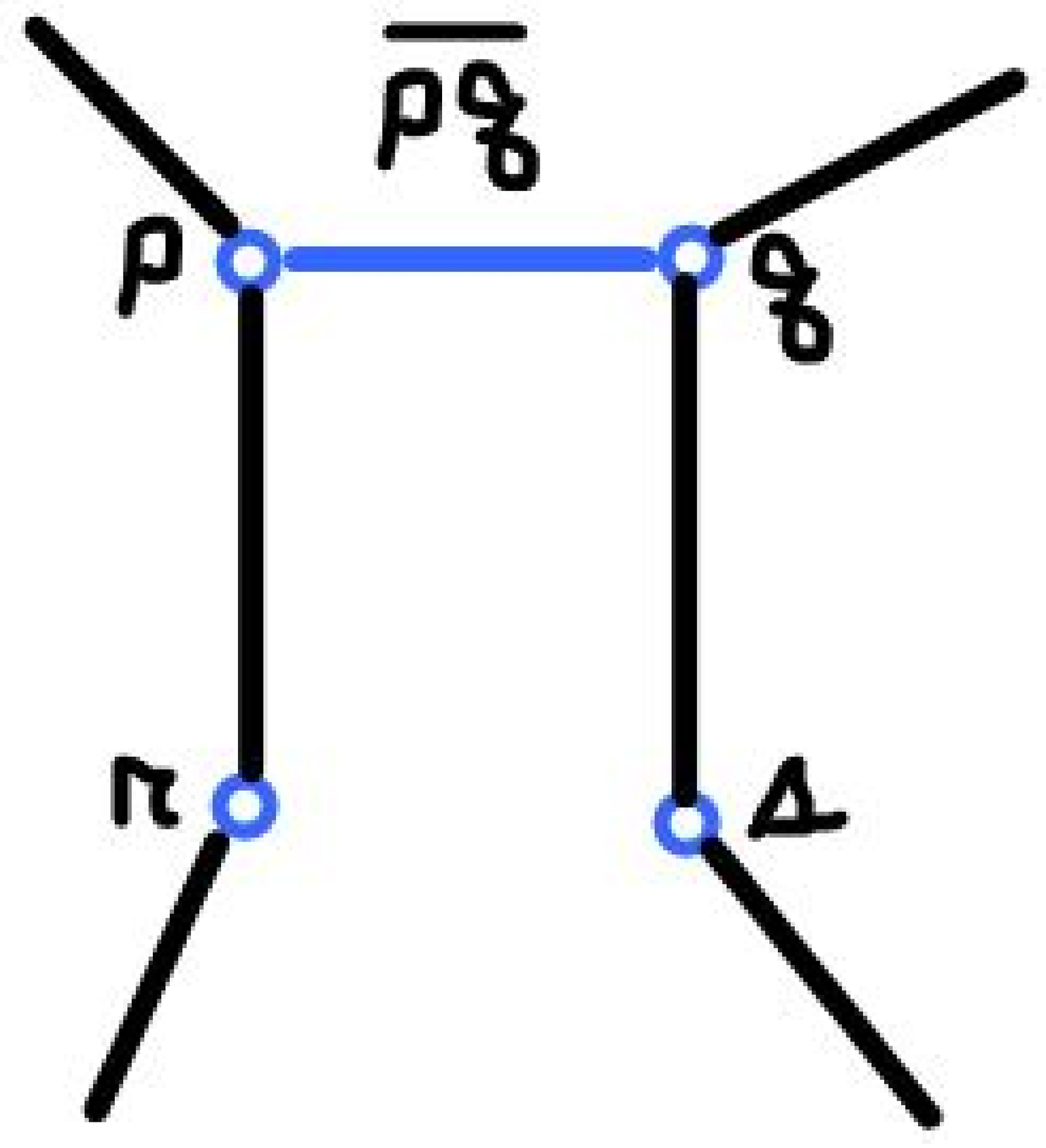}}\hfil
\subfigure[Connected subregions $\re\ r,\re\ s$]
 {\label{fig:reRS}\includegraphics[width=25mm]{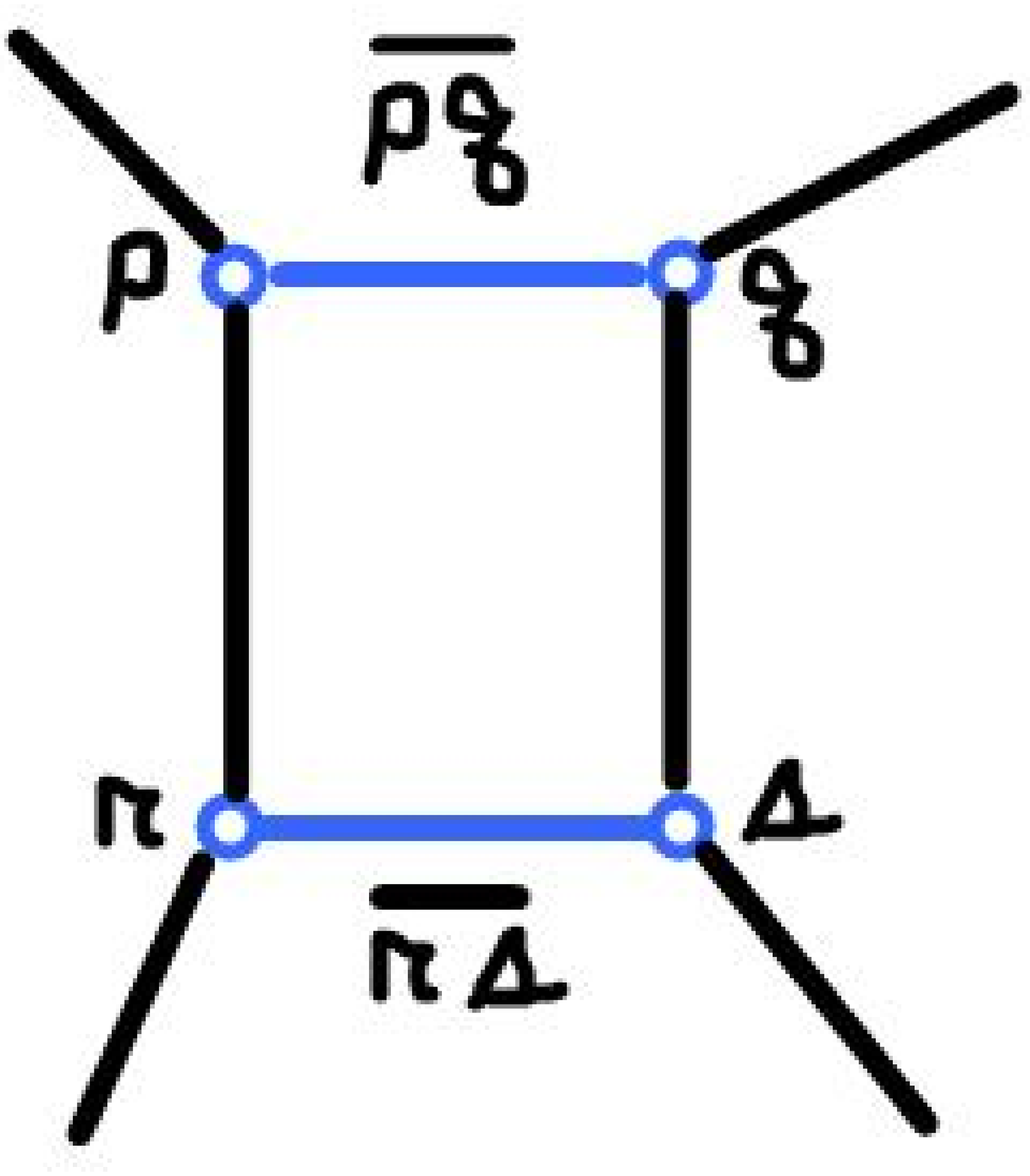}}
\caption[]{Constructing a polygon using sewing operation}
\label{fig:sewingOperation}
\end{figure}
$\mbox{}$\\
\vspace{2mm} 

\begin{example}{\bf Constructing a Polygonal Region Via Sewing}.\\
Let disconnected regions $\re A, \re B$ contain collections of subregions $\left\{\re\ p, \re\ r\right\}, \left\{\re\ q, \re\ s\right\}$, respectively, as shown in Fig.~\ref{fig:reP}.  Let $\re A = \left\{\re\ p, \re\ r\right\}$ and let $\re B = \left\{\re\ q, \re\ s\right\}$. The $sew(\re A,\re B)$ operation transforms $\re A, \re B$ into connected regions by introducing edge $\overline{p,q}$ as shown in Fig.~\ref{fig:rePQ} and edge $\overline{r,s}$ as shown in Fig.~\ref{fig:reRS}.  In other words,
\begin{align*}
sew(\re A,\re B,k) &= \left\{\left(\overline{p,q},\re A,\re B\right):\overline{p,q}\ \sn\ \re A\ \&\ \overline{p,q}\ \sn\ \re B\right\}\bigcup\\
                 &= \left\{\left(\overline{r,s},\re A,\re B\right):\overline{r,s}\ \sn\ \re A\ \&\ \overline{r,s}\ \sn\ \re B\right\}
\end{align*}
Assume that we have $\overline{p,r}\ \|\ \overline{q,s}$ (parallel edges) and assume that we also have $\overline{p,q}\ \|\ \overline{r,s}$ (parallel edges connected between regions $\re A, \re B$).  In this instance, sewing results in a rectangular shape.
\qquad \textcolor{black}{$\blacksquare$}
\end{example}

\begin{theorem}
Let $\re A, \re B$ be parallel simplicial complexes.  If $sew(\re A,\re B,k)$, then the resulting region is a simplicial complex.
\end{theorem}
\begin{proof}
Immediate from Lemma~\ref{lem:sewingOp}.
\end{proof}

\begin{theorem}
Let $\re A, \re B$ be parallel simplicial complexes.  If $sew(\re A,\re B,k)$ connects with parallel edges $\overline{p,q},\overline{r,s}$, then the simplicial complex $sew(\re A,\re B,k)$ is a rectangle.
\end{theorem}

The results obtain so far for sewing regions together extend to strings and worldsheets.  By definition, a string is simplicial complex.  Then

\begin{corollary}\label{lem:sewingOp2}
Let $\str A, \str B$ be strings.  If $sew(\str A,\str B,k)$, then the resulting region is a simplicial complex.
\end{corollary}

It is also the case that a worldsheet is a region of physical space with a covering of strings so that every subregion has nonempty intersection with the covering strings.  Hence, a worldsheet is a collection of simplexes.   

\begin{corollary}\label{lem:sewingOp3}
Let $\wsh A, \wsh B$ be strings.  If $sew(\wsh A,\wsh B,k)$, then the resulting region is a simplicial complex.
\end{corollary}

\section{Proximal Spacetime Physical Geometry}
Axiom~\ref{ax:dim} can be extended to 4-dimensional space, provided the physical geometry is defined in the spacetime of Physics.  In this section, we assume that geometrical structures are in a spacetime framework.  This means that proximal regions of physical space have relativistic mass.

Recall that a \emph{polytope} is a point set $P\subseteq R^d$ which can be presented either as a convex hull of a finite set of points or as an intersection of finitely many closed half spaces in some $R^d$~\cite{Ziegler2007polytopes}.

A region is \emph{convex}, provided a straight line can be drawn between every pair of regions so that all subregions of the line are also in the region. A \emph{hole} is a closed region $\re A$ with empty interior, {\em i.e.}, $\Int(\re A) = \emptyset$.

\begin{example}{\bf Concentric Closed Cubes with a Hole}.\\
Let $\re A, \re B$ in Fig.~\ref{fig:concentricRegions} represent concentric closed cubes.  $\re A, \re B$ are contiguous, since $\re B\subset \re A$.  The cubical interior of $\Int(\re B)$ is an example of a hole.  This hole is cube-shaped, since this hole is bounded by the surface of a cube.  From Axiom PG.~\ref{ax:area}, the hole $\Int(\re B)$ has non-zero area, {\em i.e.}, the area of $\Int(\re B)$ equals the surface area of its bounding box. \qquad \textcolor{blue}{$\blacksquare$}
\end{example}

\begin{proposition}
Let $\re A$ be a region.  Then
\begin{compactenum}[1$^o$]
\item Every line between regions is a polytope.
\item Adjacent polygons have a polytope in common.
\item A string is a polytope.
\item If a line joins the centers of mass of concentric polyhedra, then the polyhedra are path-connected.
\item A convex region contains no holes.
\item A straight line is a convex region.
\end{compactenum}
\end{proposition}
\begin{proof}$\mbox{}$\\
\begin{compactenum}[1$^o$]
\item 
We prove only the following case.  Let $\re A,\re B$ be path-connected regions, $L$ a line with one end in $\re A$ and the other end in $\re B$, $\re E\ \sn\ \re H$ a pair of connected regions in the path between $\re A,\re B$ and let $\overline{\re E, \re H}\subset L$ (a line segment in $L$).  Hence, the line $\overline{\re A, \re B}$ is defined by the path between $\re A,\re B$, which is a sequence of pairwise adjacent regions.  Hence, the line is a polytope.
\item $1\Rightarrow 2$, since adjacent polygons have a common edge.
\item Since a string is defined by a path, the result follows.
\item $1\Rightarrow 4$.
\item A convex region contains all line segments connecting any pair of its subregions.  Without loss of generality, let $\Int A = \emptyset$, a hole in closed region $\re A$.  All of the subregions of a line segment connecting any pair of subregions along the border of $\re A$ are not contained in $\re A$.  Hence, $\re A$ is not convex.
\item Immediate from the definition of straight line.
\end{compactenum}
\end{proof}

\begin{figure}[!ht]
\centering
\includegraphics[width=60mm]{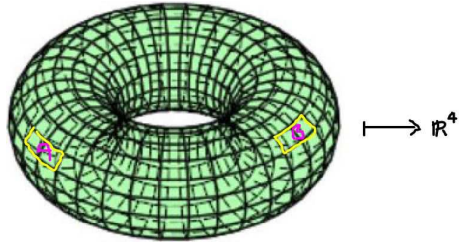}
\caption[]{Torus Regions}
\label{fig:torusmap}
\end{figure}

\begin{theorem}\label{thm:re2reBUT} {\bf Region-2-Region Based Borsuk-Ulam Theorem (re2reBUT)}{\rm \cite{PetersTozzi2016arXivWiredFriend}}.\\
Suppose that $(X, \tau_X, {\sn}_X)$, where space $X$ is $n$-dimensional and $(\mathbb{R}^{k}, \tau_{\mathbb{R}^{k}}, {\sn}_{\mathbb{R}^{k}}), k > 0$ are topological spaces endowed with compatible strong proximities.
If $f:2^X\longrightarrow \mathbb{R}^{k}$ is $\sn$-Re.s.p continuous, then $f(A) = f(\righthalfcap A)$ for some $A\in 2^X$.
\end{theorem}
$\mbox{}$\\
\vspace{3mm}

\begin{example}
Let $X$ be a set of rectangular-shaped regions on a torus represented in Fig.~\ref{fig:torusmap}.  Assume that each region $A\in X$ has the same area and colour in the RGB colour space.  Regions $A,B$ in Fig.~\ref{fig:torusmap} are examples of antipodes.  Further, let $(X, \tau_X, {\sn}_X)$, $(\mathbb{R}^{4}, \tau_{\mathbb{R}^{4}}, {\sn}_{\mathbb{R}^{4}}), k = 4$ be topological spaces endowed with compatible strong proximities.  In addition, let $f:2^X\longrightarrow \mathbb{R}^{k}$ be $\sn$-Re.s.p continuous such that $f(A) = (\area A,red,green,blue)\in \mathbb{R}^4$, where, for example, \emph{red} is the intensity of redness on the surface of $A$.  From Theorem~\ref{thm:re2reBUT}, $f(A) = f(B)$.
\qquad \textcolor{blue}{$\blacksquare$}
\end{example}

Recall that a \emph{string} is a path described by a moving particle such as a blood vessel or a chunk of matter wandering through interstellar space.  A \emph{string space} is a nonempty set of strings.  

\begin{corollary}\label{cor:stringSpace}{\rm \cite{PetersTozzi2016arXivWiredFriend}}.\\
Let $X$ be a string space.  Assume $(X, \tau_X, {\sn}_X)$, $(\mathbb{R}^k, \tau_{\mathbb{R}^k}, {\sn}_{\mathbb{R}^k})$, $k > 0$ are topological string spaces endowed with compatible strong proximities.
If $f:2^X\longrightarrow \mathbb{R}^k$ is $\sn$-Re.s.p continuous, then $f(\str A) = f(\righthalfcap \str A)$ for some string $\str A\in 2^X$.
\end{corollary}

\begin{example}
Let $X$ be a set of elliptical-shaped strings represented in Fig.~\ref{fig:strings0}.  Assume that each string $\str A\in X$ has the same curvature and colour in the RGB colour space.  Strings $\str A,\str B$ in Fig.~\ref{fig:strings0} are examples of antipodes.  Further, let 
\[
\left(X, \tau_X, {\sn}_X\right), (\mathbb{R}^{4}, \tau_{\mathbb{R}^{4}}, {\sn}_{\mathbb{R}^{4}}), k = 4
\]
 be topological spaces endowed with compatible strong proximities.  In addition, let $f:2^X\longrightarrow \mathbb{R}^{k}$ be $\sn$-Re.s.p continuous such that $f(\str A) = (\curvature(\str A),red,green,blue)\in \mathbb{R}^4$, where, for example, \emph{red} is the intensity of redness on $\str A$.  From Theorem~\ref{thm:re2reBUT}, $f(\str A) = f(\str B)$.
\qquad \textcolor{blue}{$\blacksquare$}
\end{example}

Recall that a \emph{worldsheet} is a region of space covered by strings, {\em i.e.}, every region $\re A$ in a worldsheet $\wsh M$ has nonempty intersection with at least one string $\str B\in \wsh M$.  In effect, $\re A\ \sn\ \str B$.  A \emph{worldsheet space} is a nonempty set of worldsheets.

\begin{corollary}{\rm \cite{PetersTozzi2016arXivWiredFriend}}.\\
Let $X$ be a worldsheet space.  Assume $(X, \tau_X, {\sn}_X)$, $(\mathbb{R}^k, \tau_{\mathbb{R}^k}, {\sn}_{\mathbb{R}^k})$, $k > 0$ are topological worldsheet spaces endowed with compatible strong proximities.
If $f:2^X\longrightarrow \mathbb{R}^k$ is $\sn$-Re.s.p continuous, then $f(\wsh A) = f(\righthalfcap \wsh A)$ for some worldsheet $\wsh A\in 2^X$.
\end{corollary}

\begin{theorem}\label{thm:redSnBUT}~{\rm \cite{PetersGuadagni2015strongConnectedness}}
Suppose that $(2^{S^n}, \tau_{2^{S^n}}, \snd)$ and $(\mathbb{R}^n, \tau_{\mathbb{R}^n}, \snd)$ are topological spaces endowed with compatible strong proximities.  Let $A\in 2^{S^n}$, a region in the family of regions in $2^{S^n}$.
If $f:2^{S^n}\longrightarrow \mathbb{R}^n$ is $\dnear$ Re.d.p.c. continuous, then $f(A) = f(\righthalfcap A)$ for antipodal region $\righthalfcap A\in 2^{S^n}$.
\end{theorem}

\begin{theorem}\label{cor:strBUTsimplest}{\rm \cite{PetersTozzi2016arXivWiredFriend}}.\\
Suppose that $(2^{S^n}, \tau_{2^{S^n}}, \snd)$ and $(\mathbb{R}^n, \tau_{\mathbb{R}^n}, \snd)$ are topological spaces endowed with compatible strong proximities.  Let $A\in 2^{S^n}$, a string in the family of strings in $2^{S^n}$.
If $f:2^{S^n}\longrightarrow \mathbb{R}^n$ is $\dnear$ Re.d.p.c. continuous, then $f(A) = f(\righthalfcap A)$ for antipodal string $\righthalfcap A\in 2^{S^n}$.
\end{theorem}

\begin{example}{\rm \cite{PetersTozzi2016arXivWiredFriend}}.\\
Let the Euclidean spaces $S^2$ and $\mathbb{R}^3$ be endowed with the strong proximity $\sn$ and let $\str A,\righthalfcap \str A$ be antipodal strings in $S^2$.   Further, let $f$ be a proximally continuous mapping on $2^{S^2}$ into $2^{\mathbb{R}^4}$ defined by
\begin{align*}
\str A &\in 2^{S^2},\\
\Phi &: 2^{S^2}\longrightarrow \mathbb{R}^3,\ \mbox{defined by}\\
\Phi(\str A) &= \left(\mbox{bounded,length,shape}\right)\in \mathbb{R}^3,\ \mbox{and}\\
f:2^{S^2} &\longrightarrow 2^{\mathbb{R}^4},\ \mbox{defined by}\\ 
f(2^{S^2}) &= \left\{\Phi(\str A)\in \mathbb{R}^3: \str A\in 2^{S^2}\right\}\in 2^{\mathbb{R}^4}.\mbox{\qquad \textcolor{blue}{$\blacksquare$}}
\end{align*}
\end{example}

\section{Descriptive Physical Geometry}
This section introduces the axioms for descriptive physical geometry.  
Let $X$ be a nonempty set of regions, $2^X$ the collection of all subsets of $X$, $2^{2^X}$ the collection of all collections of regions in $2^X$.  Let $\re A$ be a region in $\re A$.  Let $\phi:\re A\longrightarrow \mathbb{R}$ be a probe function that maps a region $\re A$ to a feature value in the set of reals $\mathbb{R}$. The description of region $\re A$ (denoted by $\Phi(\re A)$ is defined by
\[
\Phi(\re A) = \left(\phi_1(\re A),\dots,\phi_n(\re A)\right)\in \mathbb{R}^n\ \mbox{(feature vector describing $\re A$)}.
\]
Let $\mathscr{R}$ be a collection of regions.  The description of $\mathscr{R}$ (denoted by $\Phi(\mathscr{R})$) is defined by
\[
\Phi(\mathscr{R}) = \left\{\Phi(\re A):\re A\in \mathscr{R}\right\}\ \mbox{(Description of a Collection of Regions)}.
\]
A class of regions with a representative region $\re A$ (denoted by $\mathscr{R}_{\re A}$) is defined by
\[
\mathscr{R}_{\re A}= \left\{\re B\in 2^X: \re A\ \dnear\ \re B\right\}\ \mbox{(Class of Regions)}.
\]
In other words, a class of regions is a collection of regions that have a common description\footnote{The notion of a class of regions was suggested by Clara Guadagni}.

In addition to descriptive proximity $\dnear$, the axioms for descriptive physical geometry includes the notation $\|_{\Phi}$ (descriptively parallel).  For regions $\re A, \re B$, we write $\re A\ \|_{\Phi}\ \re B$, which reads \emph{region $\re A$ is descriptively parallel to region $\re B$}, {\em i.e.}, parts of $\re A$ and $\re B$ have matching descriptions.\\
\vspace{3mm}

\noindent{\bf Descriptive Physical Geometry Axioms}
\begin{compactenum}[{\bf Axiom d}.1]
\item\label{ax:region3} $\re A\ \snd\ \re B$ for every pair of adjacent subregions $\re A,\re B$ in the path of a line.
\item\label{ax:description} Every region has a description.
\item The description of a polytope is a sequence of descriptions of the members of the polytope.
\item A closed region $\re A$ has a description, if and only if $\Phi(\re A) = \Phi(\cl(\re A))$.
\item\label{ax:congruent} Regions are descriptively congruent, provided the regions have matching feature vectors. 
\item\label{ax:shape} $\shape(\re A)\ \dnear\ \shape(\re B)$, if and only if $\perim(\bdy(\re A)) = \perim(\bdy(\re B))$. 
\item $\shape(\re A)\ \dnear\ \shape(\re B)$, if and only if $\area(\Int(\re A)) = \area(\Int(\re B))$.
\item\label{ax:descriptivelyParallel} $\re A\ \|_{\Phi}\ \re B$, if and only if $\str A\ \sfar\ \str B\ \&\ \re A\ \dnear\ \re B$. \qquad \textcolor{blue}{$\blacksquare$}
\end{compactenum}

\begin{remark}
Axiom d.\ref{ax:shape} gets its inspiration from~\cite{Vakil2011AMMperimeter}.  Let $N_r(X)$ be a neighbourhood of a plane set $X$ with radius $r>0$, defined by
\[
N_r(X) = \left\{y: \norm{x - y} \leq r\ \mbox{for some}\ x\in X\right\}.
\]
For a shape $\shape X$, the perimeter of $\shape X$ equals $\perim\left(N_r(X)\right)$.  

Let $\shape(\re A),\shape(\re B)$ be shapes of regions $\re A,\re B$.  From axiom~d.\ref{ax:shape}, $\perim(\re A) = \perim(\re B)$ implies that $\shape(\re A)\ \dnear\ \shape(\re B)$.
\qquad \textcolor{blue}{$\blacksquare$}
\end{remark}
$\mbox{}$\\
\vspace{3mm}

\begin{figure}[!ht]
\begin{center}
\begin{pspicture}
 (0.0,0.5)(2.5,3.5)
\psframe[linecolor=black](-0.8,0.0)(4.5,3.5)
\pscircle[linecolor=black,linestyle=dotted,dotsep=0.05,fillstyle=solid,fillcolor=yellow](0.58,1.55){1.28}
\psline[linewidth = 2pt,linecolor = blue](0.58,0.55)(0.58,2.55)
\psline[linestyle=dotted,dotsep=0.05,linewidth = 2pt,linecolor = blue](0.58,2.55)(0.58,3.50)
\psline[linewidth = 2pt,linecolor = red](0.58,0.95)(0.58,1.85)
\pscircle[linecolor=black,linestyle=dotted,dotsep=0.05,fillstyle=solid,fillcolor=yellow](3.38,1.85){1.00}
\psline[linewidth = 2pt,linecolor = blue](3.38,1.05)(3.38,2.55)
\psline[linestyle=dotted,dotsep=0.05,linewidth = 2pt,linecolor = blue](3.38,2.55)(3.38,3.50)
\psline[linewidth = 2pt,linecolor = red](3.38,1.55)(3.38,2.05)
\rput(-0.5,3.2){\large  $\boldsymbol{X}$}
\rput(0.28,0.85){\footnotesize $\boldsymbol{A}$}
\rput(0.98,1.35){\footnotesize $\boldsymbol{\re\ a}$}
\rput(0.08,2.55){\large $\boldsymbol{C}$}
\rput(3.78,1.85){\footnotesize $\boldsymbol{\re\ b}$}
\rput(3.18,1.25){\footnotesize $\boldsymbol{B}$}
\rput(2.80,2.55){\large  $\boldsymbol{E}$}
 \end{pspicture}
\end{center}
\caption{\footnotesize Strongly Far Descriptively Near Parallel Lines}
\label{fig:stronglyFar3}
\end{figure}
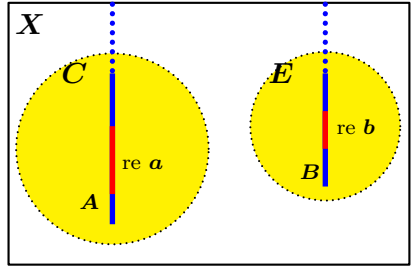
$\mbox{}$\\
\vspace{3mm}

\section{Descriptively Parallel Classes of Regions}
Axiom~\ref{ax:descriptivelyParallel} is a strong requirement for descriptively parallel regions.  Let $\re A, \re B$ be parallel, \emph{i.e.}, $\re A\ \|\ \re B$, which mean that $\re A, \re B$ satisfy Lemma~\ref{lem:parallelRegions}.  In addition, $\re A, \re B$ are descriptively parallel, provided $\re A\ \dnear\ \re B$.

\begin{example}
Assume that the lines $A,B$ in Fig.~\ref{fig:stronglyFar3} are parallel.  Let $A,B$ have features such as colour (red, green, blue or RGB).  Let $\re\ a$ be a subregion of $\re A$. A feature vector for each subregion region of $A$ is $\left(\varphi_R(\re\ a),\varphi_G(\re\ a),\varphi_B(\re\ a)\right)$, where, for example, the probe function $\varphi_R:X\longrightarrow\mathbb{R}$ is defined by $\varphi_R(\re\ a) =$ degree of redness of $\re\ a$.  Each subregion of $\re B$ has a similar feature vector.  Let $\re\ b$ be a subregion of $\re B$.  From Fig.~\ref{fig:stronglyFar3}, $\varphi_R(\re\ a) = \varphi_R(\re\ b)$.  Hence, from Axiom~\ref{ax:descriptivelyParallel}, $\re A\ \|\ \re B$.
\qquad \textcolor{blue}{$\blacksquare$}
\end{example}

Let $\mathscr{R}_{\re A},\mathscr{R}_{\re A}$ be a pair of classes of regions with representative regions $\re A,\re B$, respectively.  The classes $\mathscr{R}_{\re A},\mathscr{R}_{\re A}$ are descriptively near (denoted by $\mathscr{R}_{\re A}\ \dnear\ \mathscr{R}_{\re A}$), provided $\re\ X\ \dnear\ \re\ Y$ for some region $\re\ X\in \mathscr{R}_{\re A}, \re\ Y\in \mathscr{R}_{\re B}$.

Then we have the following parallel axiom for classes of regions.
\vspace{3mm} 
\begin{description}
\item[{\bf Axiom d.9 Parallel Classes of Regions:}]\label{ax:ParallelClasses}$\mbox{}$\\ 
Let $\re\ X\in \mathscr{R}_{\re A}, \re\ Y\in \mathscr{R}_{\re B}$.   $\mathscr{R}_{\re A}\ \|\ \mathscr{R}_{\re B}$, if and only if 
\[
\re X\ \|\ \re Y
\]
for each pair of regions in $\re X\in \mathscr{R}_{\re A}, \re Y\in \mathscr{R}_{\re B}$.
\end{description}

\begin{lemma}\label{lem:parallelClasses}
If each pair of regions $\re X\in \mathscr{R}_{\re A}, \re Y\in \mathscr{R}_{\re B}$ is expanded indefinitely and $\re X\ \sfar\ \re Y$, then $\mathscr{R}_{\re A}\ \|\ \mathscr{R}_{\re B}$.
\end{lemma}
\begin{proof}
From Lemma~\ref{lem:parallelRegions}, $\re X\ \|\ \re Y$.  Hence, from Axiom~d.\ref{ax:ParallelClasses}, $\mathscr{R}_{\re A}\ \|\ \mathscr{R}_{\re B}$.
\end{proof}

\begin{remark}
It is possible for every member of a class of regions $\mathscr{R}_{\re A}$ to have a matching feature such as colour and a non-matching feature such as shape.
\qquad \textcolor{blue}{$\blacksquare$}
\end{remark}

Axiom d.9 has a descriptive counterpart that sets forth the conditions for classes of regions to be descriptively parallel.

\vspace{3mm} 
\begin{description}
\item[{\bf Axiom d.10 Descriptively Parallel Classes of Regions:}]\label{ax:descriptivelyParallelClasses}$\mbox{}$\\ 
Let $\re\ X\in \mathscr{R}_{\re A}, \re\ Y\in \mathscr{R}_{\re B}$.   $\mathscr{R}_{\re A}\ \|_{\Phi}\ \mathscr{R}_{\re B}$, if and only if 
\[
\re X\ \|_{\Phi}\ \re Y
\]
for each pair of regions in $\re X\in \mathscr{R}_{\re A}, \re Y\in \mathscr{R}_{\re B}$.
\end{description}

\begin{theorem}
Let $\mathscr{R}_{\re A}, \mathscr{R}_{\re B}$ be classes of regions and, for each pair of regions $\re X\in \mathscr{R}_{\re A}, \re Y\in \mathscr{R}_{\re B}$ extended indefinitely, $\re X\ \sfar\ \re Y$ and $\re X\ \dnear\ \re Y$, then $\mathscr{R}_{\re A}\ \|_{\Phi}\ \mathscr{R}_{\re B}$.
\end{theorem}
\begin{proof}
From Lemma~\ref{lem:parallelClasses}, $\mathscr{R}_{\re A}\ \|\ \mathscr{R}_{\re B}$.   In addition, $\re X\ \dnear\ \re Y$for each pair $\re X\in \mathscr{R}_{\re A}, \re Y\in \mathscr{R}_{\re B}$.  Hence, from Axiom~d.10, $\mathscr{R}_{\re A}\ \|_{\Phi}\ \mathscr{R}_{\re B}$.
\end{proof}

\begin{figure}[!ht]
\begin{align*}
\begin{CD}
X @>\text{f}>> E\\
@. @VV\text{$\pi$}V\\
@. B
\end{CD}
&\qquad\qquad
\begin{CD}
\Ree A @>\text{f}>> \mathscr{R}_{\re A}\\
@. @VV\text{$\Phi$}V\\
@. B\subset \mathbb{R}^n
\end{CD}
\end{align*}
\caption[]{Two forms of fibre bundles: {\bf Spatial}: $\pi: E\ \longrightarrow\ B$
and {\bf Descriptive}: $\Phi: \mathscr{R}_{\re A}\ \longrightarrow\ B\subset\mathbb{R}^n$.}
\label{fig:projections}
\end{figure}
$\mbox{}$\\
\vspace{2mm}

\section{Two Forms of Fibre Bundles}
Classes of all half-lines with the same endpoint are called bundles in space geometry~\cite[\S 4]{Brossard1964AMMBirkhoffSpaceGeometry}.  In a descriptive proximal physical geometry, the focus shifts from
a spatial to a descriptive source of bundles.
In general, a continuous mapping $\pi:E\longrightarrow B$ is called a \emph{projection} by M. Zisman~\cite{Zisman1999fibreBundles} (also called a \emph{fibre bundle} by S.-S. Chern~\cite[\S 6, p. 683]{Chern1990AMMgeometry}), $E$ the total space, and $B$ the base space.  G. Luke and A.S. Mishchenko call $E$ the fibre space~\cite{LukeMishchenko1998VectorBundles}. For any $x\in B, \pi^{-1}(x) = e\in E$ is a \emph{fibre} of the map $\pi$.  The triple $(E,\pi,B)$ is a \emph{fibre space}, provided $E,B$ are topological spaces and $\pi$ is surjective and continuous~\cite{BreMillerSloyer1967AMMfiberSpacesAndSheaves}.

The arrow diagram 
\begin{align*}
\begin{CD}
X @>\text{f}>> E\\
@. @VV\text{$\pi$}V\\
@. B
\end{CD}
\end{align*}
$\mbox{}$\\
\vspace{3mm}

\noindent in Fig.~\ref{fig:projections} represents a \emph{spatial fibre space} $(E, \pi, B)$ in which $E,B$ are topological spaces and $\pi:E\longrightarrow B$ is surjective and continuous.  R.S. BreMiller and C.W. Sloyer define $(E, \pi, B)$ to be a \emph{sheaf}, provided $\pi:E\longrightarrow B$ is a local surjective homeomorphism~\cite{BreMillerSloyer1967AMMfiberSpacesAndSheaves}, {\em i.e.}, $\pi:E\longrightarrow B$ is a continuous, 1-1 mapping on $E$ onto $B$. For $x\in B$, $\pi^{-1}(x) = e\in E$ is a fibre over $x$ and $\pi^{-1}(B) = E$ is a fibre bundle over $B$.

Let $\Ree A$ be set of regions $\re A$, $\mathscr{R}_{\re A}$ a class of regions, $B\subset \mathbb{R}^n$, which is a subset of an $n$-dimensional feature space.  The arrow diagram$\mbox{}$\\
\vspace{2mm}
 
\begin{align*}
\begin{CD}
\Ree A @>\text{f}>> \mathscr{R}_{\re A}\\
@. @VV\text{$\Phi$}V\\
@. B\subset \mathbb{R}^n
\end{CD}
\end{align*}
$\mbox{}$\\
\vspace{2mm}

\noindent in Fig.~\ref{fig:projections} represents a descriptive fibre space $(\mathscr{R}_{\re A}, \Phi, B)$ in which $\mathscr{R}_{\re A},B$ are topological spaces and $\Phi:\mathscr{R}_{\re A}\longrightarrow B$ is surjective and continuous.

Let $\Ree A$ be a set of regions $\re A$, $\mathscr{R}_{\re A}$ a class of regions, $B\subset \mathbb{R}^n$, a subset in an $n$-dimensional feature space.  In addition, let $f:\Ree A\longrightarrow \mathscr{R}_{\re A}, \Phi:\mathscr{R}_{\re A}\longrightarrow \mathbb{R}^n$ be continuous mappings such that
\[
\Ree A\ \mathop{\longmapsto}\limits^f\ \mathscr{R}_{\re A}\ \mathop{\longmapsto}\limits^{\Phi}\ B\subset\mathbb{R}^n.
\]
For region $\re A\in \mathscr{R}_{\re A}$, recall that $\mathscr{R}_{\re A}\ \mathop{\longmapsto}\limits^{\Phi}\ \mathbb{R}^n$ is defined by
\[
\Phi(\re A) = \left(\varphi_1(\re A),\dots,\varphi_i(\re A)\dots,\varphi_n(\re A)\right): \varphi_i(\re A)\in \mathbb{R}.
\]
$\Phi(\re A)$ is a feature vector that provides a \emph{description} of region $\re A$. In a descriptive proximal physical geometry, $f$ maps a collection of regions $\Ree A$ to a class of regions $\mathscr{R}_{\re A}$, which is projected onto $B\subset \mathbb{R}^n$ by $\Phi$.  In effect, $\Phi\left(f(\Ree A)\right)\in \mathbb{R}^n$.  

Let $\re X\in \mathscr{R}_{\re A}$ and let $\Phi(\re X)$ be a feature vector that describes region $\re X$ in the class of regions $\mathscr{R}_{\re A}$.  The continuous mapping $\Phi:\mathscr{R}_{\re A}\longrightarrow B$ is a \emph{projection} on the class of regions $\mathscr{R}_{\re A}$ onto $B$ such that $(\mathscr{R}_{\re A}, \Phi, B)$ is a \emph{descriptive fibre space} and the set of feature vectors $B\subset \mathbb{R}^n$ is the \emph{descriptive base space}.  For $x$ in $B$, $\Phi^{-1}(x)\in \mathscr{R}_{\re A}$ is a fibre over $x$.  The leads to a descriptive fibre bundle which is a BreMiller-Sloyer sheaf for a class of regions.  

\begin{theorem}\label{thm:fibreBundle}
Let $\re X\in \mathscr{R}_{\re A}$ and let $\Phi(\re X)$ be a feature vector for region $\re X$ in the class of regions $\mathscr{R}_{\re A}$, $x\in\mathbb{R}^n$.  Assume that every $\re X$ in $\mathscr{R}_{\re A}$ has a unique feature, $\Phi:\mathscr{R}_{\re A}\longrightarrow B\subset \mathbb{R}^n$ is a local surjective homeomorphism in which $\mathscr{R}_{\re A}, B$ are topological spaces and $\Phi^{-1}(x)\in \mathscr{R}_{\re A}$ for $x\in B$.  Then 
\begin{compactenum}[1$^o$]
\item $(\mathscr{R}_{\re A}, \Phi, B)$ is a BreMiller-Sloyer sheaf.
\item $\Phi^{-1}(B)$ is a descriptive fibre bundle over $B$.
\end{compactenum}
\end{theorem}
\begin{proof}$\mbox{}$\\
1$^o$: $(\mathscr{R}_{\re A}, \Phi, B)$ is, by definition, a BreMiller-Sloyer sheaf.\\
2$^o$: Let $\re X\in \mathscr{R}_{\re A}$ and let $x\in \mathbb{R}^n$ be a feature vector that describes $\re X$.  Without loss of generality, assume each $\re X$ has a unique shape.  $\Phi^{-1}(x) = \re X$, since every $\re X$ has a unique shape in the class of regions $\mathscr{R}_{\re A}$.  Hence,  $\Phi^{-1}(B) = \mathscr{R}_{\re A}$ is a descriptive fibre bundle over $B$.
\end{proof}

\begin{example}
Let $sew(\re A,\re B, k)$ be a simplicial complex (from Lemma~\ref{lem:sewingRegions}) derived from a pair of disjoint regions $\re A,\re B$, $\mathscr{R}_{sew(\re A,\re B, k)}$ the class of all planar shapes that are similar to $sew(\re A,\re B, k)$, and $B$ is a set of descriptions of planar shapes.   
The arrow diagram 
\begin{align*}
\begin{CD}
\left\{re A,\re B\right\} @>\text{sew}>> \mathscr{R}_{sew(\re A,\re B, k)}\\
@. @VV\text{$\Phi$}V\\
@. B\subset \mathbb{R}^n
\end{CD}
\end{align*}
$\mbox{}$\\
\vspace{2mm}

\noindent represents a descriptive fibre space 
$
(\mathscr{R}_{sew(\re A,\re B, k)}, \Phi, B)
$
in which $\mathscr{R}_{sew(\re A,\re B, k)},B$ are topological spaces and $\Phi:\mathscr{R}_{sew(\re A,\re B, k)}\longrightarrow B$ is a homeomorphism.  From Theorem~\ref{thm:fibreBundle}, 
$
(\mathscr{R}_{sew(\re A,\re B, k)}, \Phi, B)
$
is a BreMiller-Sloyer sheaf and $\Phi^{-1}(B)$ is a descriptive fibre bundle over $B$.
\qquad \textcolor{blue}{$\blacksquare$}
\end{example}

Let $\mathscr{R}_{\re A}, \mathscr{R}_{\re A'}$ be a pair of classes of regions.  In addition, let $X\in \mathscr{R}_{\re A}$ (a region in class $\mathscr{R}_{\re A}$), $Y\in \mathscr{R}_{\re A'}$ (a region in class $\mathscr{R}_{\re A'}$).  Assume $X\ \sfar\ Y$ ($X$ and $Y$ are strongly far apart) for each pair of regions in $\mathscr{R}_{\re A}, \mathscr{R}_{\re A'}$, respectively. Then, from Lemma~\ref{lem:parallelRegions}, $\mathscr{R}_{\re A}, \mathscr{R}_{\re A'}$ are parallel classes.  That is, $\mathscr{R}_{\re A}\ \|\ \mathscr{R}_{\re A'}$ are parallel, if and only if $X\ \|\ Y$ for all $X\in \mathscr{R}_{\re A},Y\in \mathscr{R}_{\re A'}$.

\begin{theorem}\label{thm:parallelClasses}
Let $B,H\subset \mathbb{R}^n$ be sets of feature vectors that describe regions in the classes of regions $\mathscr{R}_{\re A}, \mathscr{R}_{\re A'}$, respectively. Assume that every region $\re X\in \mathscr{R}_{\re A}, \re Y\in \mathscr{R}_{\re A'}$ has a unique feature and that $\re X\ \sfar\ \re Y$.  Also, let $\Phi_1:\mathscr{R}_{\re A}\longrightarrow B\subset \mathbb{R}^n, \Phi_2:\mathscr{R}_{\re A'}\longrightarrow H\subset \mathbb{R}^n$ be homeomorphisms and let $\Phi^{-1}(B),\Phi^{-1}(H)$ be descriptive fibre bundles.  Classes $\mathscr{R}_{\re A},\mathscr{R}_{\re A'}$ are parallel if and only if $\Phi^{-1}(B)\ \|\ \Phi^{-1}(H)$.
\end{theorem}
\begin{proof}
Since each feature vector in $X$ is unique in its description of a region in $\mathscr{R}_{\re A}$, we have $\Phi_1^{-1}(B) = \mathscr{R}_{\re A}$.  Similarly, $\Phi_2^{-1}(H) = \mathscr{R}_{\re A'}$.  From the fact that $\Phi_1^{-1}(b)\ \sfar\ \Phi_1^{-1}(h), b\in B, h\in H$ for each pair of regions,  by Lemma~\ref{lem:parallelRegions}, classes $\mathscr{R}_{\re A},\mathscr{R}_{\re A'}$ are parallel.  Hence,
$\mathscr{R}_{\re A}\ \|\ \mathscr{R}_{\re A} \Leftrightarrow \Phi^{-1}(B)\ \|\ \Phi^{-1}(H)$.
\end{proof}

The result from Theorem~\ref{thm:parallelClasses} is similar to the result for descriptively parallel classes in Theorem~\ref{thm:descriptivelyparallelClasses}.

\begin{theorem}\label{thm:descriptivelyparallelClasses}
Let $X,Y\subset \mathbb{R}^n$ be sets of feature vectors that describe regions in the classes of regions $\mathscr{R}_{\re A}, \mathscr{R}_{\re B}$, respectively.
Assume that every region $\re X\in \mathscr{R}_{\re A}, \re Y\in \mathscr{R}_{\re B}$ has a unique feature and that $\re X\ \sfar\ \re Y$.
Let $\Phi_1:\mathscr{R}_{\re A}\longrightarrow \mathbb{R}^n, \Phi_2:\mathscr{R}_{\re B}\longrightarrow \mathbb{R}^n$ be continuous mappings.  Classes $\mathscr{R}_{\re A}, \mathscr{R}_{\re A}$ are descriptively parallel fibre bundles, if and only if $\Phi^{-1}(X)\ \|_{\Phi}\ \Phi^{-1}(Y)$.
\end{theorem}
\begin{proof}
The result follows from Axiom d.\ref{ax:descriptivelyParallel} with the proof symmetric with the proof of Theorem~\ref{thm:parallelClasses}.
\end{proof}

\section{Duality}
Recall that all propositions in projective geometry occur in dual pairs, {\em i.e.}, starting from either a proposition about a pair of points or a pair of lines, another proposition results by replacing points with lines or lines with points.  Since we are dealing with vertex regions instead of points in physical geometry, propositions about vertices hold true for propositions about non-vertex regions such as lines and strings.

Here is a well-known theorem rewritten for physical geometry, introduced by Pascal in 1640~\cite{Weisstein2016PascalTheorem}.

\begin{theorem}\label{thm:Pascal}{\bf Pascal's Theorem{\bf\cite{Weisstein2016BrianchonTheorem}}}
For a hexagon inscribed in a conic section, three pairs of the continuations of the opposite sides meet on a straight line.
\end{theorem}

The dual of Theorem~\ref{thm:Pascal} is Theorem~\ref{thm:Casey}, known as Brianchon's Theorem, introduced by J. Casey in 1888.

\begin{theorem}\label{thm:Casey}{\bf Brianchon's Theorem{\bf\cite{Casey1888PascalDual}}}
A hexagon circumscribed on a conic section has the lines joining opposite polygon vertices (polygon diagonals) meet in a single subregion.
\end{theorem}

\begin{conjecture}
Theorem~\ref{thm:Casey} has descriptive proximity form.  
\end{conjecture}
\begin{proof}
Let $p_1,\dots,p_8$ be vertices of hexagon $H$ circumscribing a conic section and let $\overline{p_1,p_5},\dots,\overline{p_8,p_4}$ be line segments with extremities that are opposite vertices in $H$.  Also, let $I$ be the index set for the vertices of $H$ and let $\re A$ be one of line segments joining opposite vertices in $H$.  And $\varepsilon\left(\overline{p_i,p_j}\right)$ equal length. Then $\overline{p_i,p_j}$ meet in the same subregion $\re C$,  $\re A\ \sn\ \overline{p_i,p_j}$ and $\re A\ \dnear\ \overline{p_i,p_j}$ for every $\overline{p_i,p_j}$.
\end{proof}

There are many other forms of duality in physical geometry.

\begin{remark}{\bf Duality in Physical Geometry}
A consequence of Axiom d.\ref{ax:congruent} is a natural duality among geometrical structures.  
\begin{compactenum}[{\bf dual}.1]
\item On a hypersphere $S^n$, the dual of a region $\re A\in S^n$ is its antipode $\righthalfcap \re A$ with a description that matches the description of $\re A$.
\item For a line segment $L$, its dual is a region of another line parallel with $L$.
\item The dual of a hole is a region with no holes.
\item\label{duality:spinningParticle} The dual of a spinning particle is a massless spin particle~\cite[\S 1.2.1, p. 14]{GreenScharzWitten2012CUPsuperstrings}.
\item\label{duality:boson} The dual of a boson in 6 dimensions is a fermion in 10 dimensions~\cite[\S 1.2.2, p. 16]{GreenScharzWitten2012CUPsuperstrings}.
\item\label{duality:gaugeTheory} The dual of gauge theory in 4 dimensions is string theory in 10 dimensions~\cite[p. 567]{Boschi-FilhoBraga2007gaugeStringDuality}.
\end{compactenum}
\qquad \textcolor{blue}{$\blacksquare$}
\end{remark}

\begin{conjecture}
Duality~\ref{duality:spinningParticle} has descriptive proximity form.  Let $\re A, \re B$ be spinning particle and massless spin particle, respectively.  $\re A\ \dnear\ \re B$.
\end{conjecture}
\begin{proof}
From Axiom~\ref{ax:description}, every region of space such as the region of space occupied by a particle has a description.  A feature common to $\re A, \re B$ is spin, which is treated as a Boolean variable, {\em i.e.}, 
\[
\varphi_{spin}(\re A) = \varphi_{spin}(\re B) = true.
\]
Another feature common to $\re A, \re B$ is mass, which is also treated as a Boolean variable, {\em i.e.}, 
\[
\varphi_{mass}(\re A) = \mbox{true}, \varphi_{mass}(\re B) = false.
\]
Hence, $\re A\ \dnear\ \re B$, since the two regions (space occupied by spinning particles) have at least one matching feature value.
\end{proof}

\begin{conjecture}
Duality~\ref{duality:boson} has descriptive proximity form.  Let $\re A, \re B$ be boson and fermion, respectively.  $\re A\ \dnear\ \re B$.
\end{conjecture}

\begin{conjecture}\label{conj:gaugeField}
Duality~\ref{duality:gaugeTheory} has descriptive proximity form.  
\end{conjecture}
\begin{remark}
The proof of Conjecture~\ref{conj:gaugeField} can be derived from superstring theory given by M.B. Green, J.H. Schwarz and E. Witten in~\cite[\S 12.2-12.3, pp. 281-283]{GreenScharzWitten2012CUPsuperstrings}.
\end{remark}

\section{Results for Descriptive Physical Geometry}
\noindent This section gives a selection of results for descriptive physical geometry.

\begin{proposition}
If $\re B$ is in a class of regions $\mathscr{R}_{\re A}$, then $\Phi(\re B)\in \Phi(\mathscr{R}_{\re A})$.
\end{proposition}
\begin{proof}
Immediate from the definition of $\Phi(\mathscr{R}_{\re A})$.
\end{proof}

\begin{proposition}\label{prop:interior}
Let $\Int(\re A)$ be the interior of a region $\re A$.  Then $\Phi(\Int(\re A))\in \mathbb{R}^n$.
\end{proposition}
\begin{proof}
From Axiom~d.\ref{ax:description}, $\re A$ has a description and, by definition, $\Phi(\Int(\re A))\in \mathbb{R}^n$.
\end{proof}

\begin{example}{\bf Class of Nome Regions}.\\
The inverse elliptic nome $(\overline{q})_{\infty}$ is defined by
\begin{align*}
\overline{q} &\equiv e^{2\pi i \tau}\mbox{(square nome)}.\\
\tau &= \frac{\omega_2}{\omega_1} \mbox{(two half periods of an elliptic function)}.\\
(\overline{q})_{\infty} &= \mathop{\prod}\limits_{k=0}^{\infty}\left(1 - q^k\right) \mbox{($q$ series)}.\\
\eta(\tau) &\equiv \overline{q}^{\frac{1}{24}}(\overline{q})_{\infty}\mbox{(Dedekind $\eta$(eta) function)}.\\
m(q) &= \frac{16\eta^8\left(\frac{1}{2}\tau\right)\eta^{16}(2\tau)}{\eta(\tau)}\mbox{(inverse nome)}.
\end{align*}
The inverse elliptic nome $m(q)$ spirals as shown in Fig.~\ref{fig:strings0} to form a class regions $\mathscr{R}_{\re A}$ defined by the signature region $\re A$, {\em i.e.}, $\mathscr{R}_{\re A}$ is a collection of spiral regions that are strongly descriptively near $\re A$.  For example, the curvature of the Spiral region $\re B$ in Fig.~\ref{fig:strings0} has the same description as $\re A$, if we consider the values of $e^{\pi i\tau}$.
\qquad \textcolor{blue}{$\blacksquare$}
\end{example}

\begin{proposition}
A hole has a description.
\end{proposition}
\begin{proof}
A hole is the interior of a closed region.  Hence, from Prop.~\ref{prop:interior}, the result follows.
\end{proof}

\begin{figure}[!ht]
\centering
\subfigure[Colour cube $\re B$]
 {\label{fig:concentricRegions0}\includegraphics[width=50mm]{concentricRegions}}\hfil
\subfigure[Colour cube $\re B$]
 {\label{fig:concentricRegions2}\includegraphics[width=50mm]{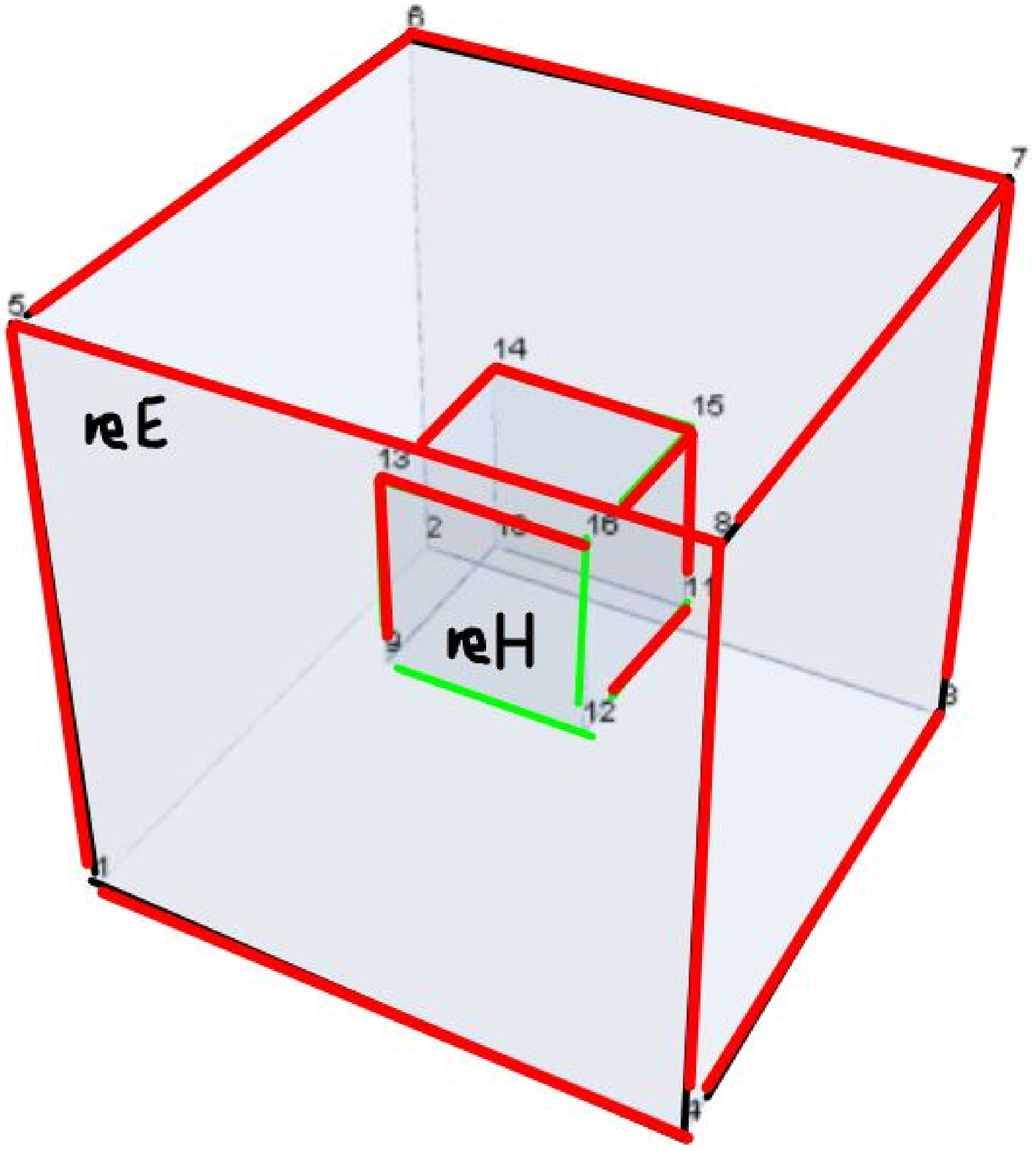}}
\caption[]{Closed Concentric Regions with Coloured Edges}
\label{fig:colouredRegions}
\end{figure}
$\mbox{}$\\
\vspace{2mm}

\begin{example}
A pair of cubes $\re A, \re E$ are represented in Fig.~\ref{fig:colouredRegions}.  Cubes $\re A, \re B$ are concentric and region $\re B$ is a closed region that has an empty interior.  Similarly, Cubes $\re E, \re H$ are concentric and region $\re H$ is a closed region that has an empty interior.  Let edge colour be in the red, green, blue (rgb) spectrum.  Since both $\re B$ and $\re H$ have green edges, we have $\re A\ \dnear\ \re E$.
\qquad \textcolor{blue}{$\blacksquare$}
\end{example}

Let string $\str A$ be a member of a region $\re A$.  In that case, $\str A$ is called a wired friend of $\re A$.  Next consider a form of the Brouwer fixed point theorem in which the shape of $\str A$ (denoted by $\shape A$) with $k$ features is mapped to a fixed point in $\mathbb{R}^k$.  Assume $\re A \subseteq \shape A$.

\begin{theorem}\label{thm:wiredFriends}{\bf String Wired Friend Theorem}.\\
Every occurrence of a wired friend $\str A$ with a particular $\strShape$ with $k$ features on $S^n$ maps to a fixed description $g(\str A)$ that belongs to a ball $B_{r,k}$ with radius $r > 0$ in $\mathbb{R}^k$.
\end{theorem}
\begin{proof}
Let string $\str A$ be a member of a region $\re A$ with shape $\shape A$ with $k$ features.  From Axiom~d.\ref{ax:description}, $\shape A$ has a description, which is a feature vector.  Also, let $\varphi:\shape A\longrightarrow \mathbb{R}$ map $\shape A$ to a feature value in $\mathbb{R}$.  Then let
\[
f: 2^{\re A}\longrightarrow \mathbb{R}^k\ \mbox{defined by}\ f(\shape A) = \left(\varphi_1(\shape A),\cdots,\varphi_k(\shape A)\right).
\]
$f(\shape A)$ is a fixed point in $\mathbb{R}^k$.   Let $B_{r,k}(X)$ be a neighbourhood of a plane set $X$ with radius $r > 0$, defined by
\[
B_{r,k}(X) = \left\{y\in \mathbb{R}^k: \norm{f(\shape A) - y} \leq r\ \mbox{for some}\ f(\shape A)\in X\right\}.
\]
Without loss of generality, let $k = 1$ for a $\shape A$ be limited to its perimeter.  Since perimeter is fixed for each $\shape A$ and $f(\shape A)\in B_{r,k}(X)$, we get the desired result.
\end{proof}

Let $\wsh A$ be a worldsheet that results from a collection of strings that wrap around a particular volume of space $A$.  That is, $\wsh A$ conforms to a spatial region.  Hence, $\wsh A$ will have a particular shape $\wshShape$.

\begin{theorem}\label{thm:wiredFriends2}{\bf Worldsheet Wired Friend Theorem}\\
Every occurrence of a wired friend $\wsh A$ with a particular $\wshShape$ with $k$ features on $S^n$ maps to a fixed description $g(\wsh A)$ that belongs to a ball $B_{r,k}$ with radius $r$ in $\mathbb{R}^k$.
\end{theorem}

\bibliographystyle{amsplain}
\bibliography{NSrefs}

\end{document}